\documentclass{article}
\usepackage[utf8]{inputenc}

\usepackage{amsmath,amssymb,amsthm}
\usepackage{bm}
\usepackage{algorithm,algorithmic}
\usepackage{enumerate}
\usepackage{color}
\usepackage{multirow}
\usepackage{graphicx}
\usepackage[margin=25mm]{geometry}
\usepackage[numbers]{natbib}
\usepackage{mathtools}
\mathtoolsset{showonlyrefs}
\usepackage[colorlinks=true,linkcolor=blue,urlcolor=black]{hyperref}

\newcommand{\argmin}{\mathop{\rm argmin}\limits}
\newcommand{\dom}{\mathop{\rm dom}}
\newcommand{\interior}{\mathop{\rm int}}

\newcommand{\convhull}{\mathop{\rm conv}}
\newcommand{\prox}{\mathrm{prox}}
\newcommand{\dist}{\mathrm{dist}}

\newcommand{\grad}{\mathrm{grad}}
\newcommand{\tr}{\mathrm{tr}}
\newcommand{\EE}{\mathbb{E}}
\newcommand{\RR}{\mathbb{R}}
\newcommand{\MM}{\mathcal{M}}
\newcommand{\OO}{\mathcal{O}}
\newcommand{\CC}{\mathcal{C}}
\renewcommand{\SS}{\mathcal{S}}
\newcommand{\innerprod}[2]{\langle #1,#2 \rangle}

\newtheorem{theorem}{Theorem}[section]
\newtheorem{corollary}{Corollary}[section]
\newtheorem{lemma}{Lemma}[section]

\newtheorem{remark}{Remark}[section]

\newtheorem{assumption}{Assumption}[section]
\counterwithin{figure}{section}
\counterwithin{table}{section}

\providecommand{\keywords}[1]
{
  \small	
  \textbf{\textit{Keywords---}} #1
}

\title{Simple linesearch-free first-order methods for nonconvex optimization}
\author{Shotaro Yagishita\thanks{Risk Analysis Research Center, The Institute of Statistical Mathematics, Japan, E-mail: syagi@ism.ac.jp} \thanks{Center for Social Data Structuring, Joint Support-Center for Data Science Research, Japan}
\and
Masaru Ito\thanks{Department of Mathematics, College of Science and Technology, Nihon University, Japan. E-mail: ito.masaru@nihon-u.ac.jp}}
\date{\today}

\begin{document}

\maketitle

\begin{abstract}
This paper presents an auto-conditioned proximal gradient method for nonconvex optimization.
The method determines the stepsize using an estimation of local curvature and does not require any prior knowledge of problem parameters and any linesearch procedures.
Its convergence analysis is carried out in a simple manner without assuming the convexity, unlike previous studies.
We also provide convergence analysis in the presence of the Kurdyka--\L ojasiewicz property, adaptivity to the weak smoothness, and the extension to the Bregman proximal gradient method.
Furthermore, the auto-conditioned stepsize strategy is also applied to the conditional gradient (Frank--Wolfe) method and the Riemannian gradient method.
\end{abstract}


\section{Introduction}\label{sec:intro}
In this paper, we consider a nonconvex optimization problem minimizing the sum $f(x)+g(x)$ of a smooth function $f$ and a nonsmooth function $g$ over a finite dimensional Euclidean space.
First-order methods such as the proximal gradient \cite{fukushima1981generalized,passty1979ergodic,lions1979splitting} or the conditional gradient (Frank--Wolfe) \cite{frank1956algorithm,levitin1966constrained,mine1981minimization,bredies2009generalized} methods are active research interests in nonconvex/convex optimization which have wide applications in machine learning, statistics, and signal processing, see, e.g., \cite{parikh2014proximal,beck2017first,braun2022conditional}.

The parameter tuning, in particular, the stepsize selection, is a central issue for the implementation affecting the performance of first-order methods. To ensure an ideal convergence property, the stepsize selection typically requires the problem-dependent knowledge such as the Lipschitz constant of $\nabla f$. The backtracking linesearch is a widely used approach to automatically estimate the Lipschitz constant (see, e.g., \cite{bertsekas2016nonlinear,beck2017first}). A bottleneck of the backtracking scheme is multiple evaluations of the objective or subproblems by retrying the iteration in order to ensure a successful estimate of a Lipschitz constant.

The stepsize choice proposed by \citet{malitsky2020gradient} lead research attentions to the ``linesearch-free'' scheme in  gradient methods \cite{malitsky2024adaptive,latafat2024proximal,oikonomidis24adaptive,ou2025bregman,ansarionnestam2025riemannian,li2025simple,lan2024projected,Hoai2024nonconvex,Hoai2025proximal}.
These methods exploit the Lipschitz constant estimate, such as $\frac{\|\nabla f(x^k)-\nabla f(x^{k-1})\|}{\|x^k-x^{k-1}\|}$ using the past test points, to update the stepsizes without retrying the iteration in contrast to backtracking.
This kind of strategy of the stepsize choice is also referred to as ``auto-conditioning'' \cite{lan2024projected,li2025simple}.
Importantly, the linesearch-free methods in the literature possess adaptive convergence behavior in theory and remarkable numerical performance in practice.

The concept of the auto-conditioning strategy originates from the pioneer works by Malitsky \cite{malitsky2015variational,malitsky2020golden} for variational inequalities, which was later applied to the steepest descent method for unconstrained smooth convex optimization by \citet{malitsky2020gradient}. This lead further extensions to proximal gradient \cite{malitsky2024adaptive,latafat2024proximal,oikonomidis24adaptive}, Bregman proximal gradient \cite{ou2025bregman}, and Riemannian gradient methods \cite{ansarionnestam2025riemannian} in smooth convex optimization. Interestingly, \citet{oikonomidis24adaptive} showed that the linesearch-free method \cite{latafat2024proximal} is ``universal'' in the sense that it enjoys adaptive convergence rates not only for Lipschitz continuity but also for H\"older continuity of $\nabla f$.
Linesearch-free proximal gradient methods with optimal complexity were also established by \citet{li2025simple}  under H\"older continuity of $\nabla f$.
Note that the aforementioned linesearch-free gradient methods assume the convexity of the objective function. The development of linesearch-free methods with weakly convex $f$ was recently addressed in \cite{lan2024projected,Hoai2024nonconvex,Hoai2025proximal}.

In this paper, we provide simple convergence analyses of linesearch-free first-order methods for nonconvex optimization problems.
Our linesearch-free first-order methods is based on the auto-conditioned stepsize strategy proposed by \citet{lan2024projected}.

We first propose an auto-conditioned proximal gradient method (AC-PGM) for composite problems whose objective function is the sum of a smooth function and a nonsmooth function, and provide its convergence analyses under the Lipschitz continuity of the gradient of the smooth term.
The AC-PGM does not require any prior knowledge of problem parameters and any linesearch procedures.
Existing works \cite{lan2024projected,Hoai2024nonconvex,Hoai2025proximal} imposed the convexity on the nonsmooth term and several conditions, which does not required in this paper.
We further establish a convergence result in the presence of the Kurdyka--\L ojasiewicz (KL) property for the AC-PGM.
To the best of our knowledge, this is the first result under the KL property for the linesearch-free first-order methods.
Surprisingly, it is also shown that the AC-PGM is adaptive to weak smoothness, namely, even if the Lipschitz continuity of the gradient is replaced by the H\"older continuity, the method still achieves a convergence rate adapted to the corresponding H\"older exponent.
Such analyses had been conducted by \citet{oikonomidis24adaptive} and \citet{li2025simple} in the context of linesearch-free proximal gradient methods under the convexity, but this is the first for nonconvex optimization.
Furthermore, it is also shown that the AC-PGM can be extended to the Bregman proximal gradient method.

Next, to demonstrate the generality of the auto-conditioned stepsize strategy, we propose linesearch-free first-order methods for other settings.
Specifically, auto-conditioned conditional gradient method (AC-CGM) and auto-conditioned Riemannian gradient method (AC-RGM) are considered.
Although the analyses for each algorithm slightly differ from that of the proximal gradient method, they share the common essential principle.
To the best of our knowledge, a linesearch-free conditional gradient method is proposed for the first time.
Regarding Riemannian gradient methods, \citet{ansarionnestam2025riemannian} developed an linesearch-free method; however, their analysis assumes the geodesic convexity.
In particular, on connected compact Riemannian manifolds, geodesically convex functions must be constant (see, e.g., \cite[Corollary 11.10]{boumal2023introduction}), and hence their applicability is limited.
In contrast, our convergence analysis for AC-RGM does not impose the geodesic convexity, and therefore the aforementioned restriction does not arise.

The rest of this paper is organized as follows.
The remainder of this section is devoted to notation.
In the next section, we introduce the AC-PGM and provide simple convergence analysis.
Convergence analysis in the presence of the KL property, adaptivity to the weak smoothness, and the extension to the Bregman proximal gradient method are also discussed.
Section \ref{sec:other-FoM} is devoted to other linesearch-free first-order methods.
Numerical experiments to demonstrate the performance of the auto-conditioned methods are reported in Section \ref{sec:numerical}.
Finally, Section \ref{sec:conclusion} concludes the paper with some remarks.

\subsection{Notation}
For a positive integer $n$, the set $[n]$ is defined by $[n]\coloneqq\{1,\ldots,n\}$.
We use $|S|$ to denote the cardinality of a finite set $S$.
Let $\EE$ be a finite-dimensional inner product space endowed with an inner product $\innerprod{\cdot}{\cdot}$.
The induced norm is denoted by $\|\cdot\|$.
Given a matrix $X$, $X^\top$ denotes the transpose of $X$.
$I$ denotes the identity matrix of appropriate size.
The trace of a square matrix $X$ is denoted by $\tr(X)$.
For a subset $\CC\subset\EE$, its interior and its convex hull are denoted by $\interior\CC$ and $\convhull\CC$, respectively.
The domain of a function $\phi:\EE\to(-\infty,\infty]$ is denoted by $\dom\phi\coloneqq\{x\in\mathbb{E}\mid\phi(x)<\infty\}$.

\section{Auto-conditioned proximal gradient method}
We consider a linesearch-free proximal gradient method for the following composite optimization problem
\begin{equation}\label{problem:composite}
    \underset{x\in\EE}{\mbox{minimize}} \quad F(x)\coloneqq f(x)+g(x),
\end{equation}
where $f:\EE\to(-\infty,\infty]$ is lower semicontinuous and continuously differentiable on an open set including $\dom g$, $g:\EE\to(-\infty,\infty]$ is a proper and lower semicontinuous function, and $F$ is bounded from below, namely, $F^*\coloneqq \inf_{x\in\EE}F(x)>-\infty$.

For $x\in\dom g=\dom F$,
\begin{equation}
    \widehat{\partial}F(x) \coloneqq \left\{\xi\in\mathbb{E}~\middle|~\liminf_{y\to x}\frac{F(y)-F(x)-\innerprod{\xi}{y-x}}{\|y-x\|}\ge0\right\}
\end{equation}
is called the Fr\'echet subdifferential of $F$ at $x$ and
\begin{equation}
    \partial F(x) \coloneqq \left\{\xi\in\mathbb{E}~\middle|~\exists\{x^k\},\{\xi^k\}~ \mbox{s.t.}~ x^k\to x,~ F(x^k)\to F(x),~ \xi^k\to \xi,~ \xi^k\in\widehat{\partial}F(x^k)\right\}
\end{equation}
is known as the limiting subdifferential of $F$ at $x$, where $\widehat{\partial}F(x)\coloneqq\emptyset$ for $x\notin\dom g$.
Clearly, $\widehat{\partial}F(x)\subset\partial F(x)$ holds.
Any local minimizer $x^*\in\dom g$ of \eqref{problem:composite} satisfies $0\in\widehat{\partial}F(x^*)\subset\partial F(x^*)$.
We call a point $x^*\in\dom g$ satisfying $0\in\partial F(x^*)$ an l-stationary point of \eqref{problem:composite}.
Due to the continuous differentiability of $f$, it holds that $\widehat{\partial}F(x)=\nabla f(x)+\widehat{\partial}g(x)$ and $\partial F(x)=\nabla f(x)+\partial g(x)$ for $x\in\dom g$ \citep[Exercise 8.8]{rockafellar2009variational}.

For the well-definedness of the proximal gradient method, the following is also assumed.

\begin{assumption}\label{assume:AC-PGM}
For any $\gamma>0,~ x\in\EE$,
\begin{equation}
    \prox_{\frac{g}{\gamma}}(x)\coloneqq\argmin_{y\in\EE}\left\{g(y)+\frac{\gamma}{2}\|y-x\|^2\right\}
\end{equation}
is nonempty.
\end{assumption}

Let $\gamma>0$, $x\in\dom g$, and
\begin{equation}\label{eq:prox-grad-map}
    x^+\in\prox_{\frac{g}{\gamma}}\left(x-\frac{1}{\gamma}\nabla f(x)\right)=\argmin_{y\in\EE}\left\{\innerprod{\nabla f(x)}{y}+\frac{\gamma}{2}\|y-x\|^2+g(y)\right\},
\end{equation}
then we define
\begin{equation}
    R_\gamma(x)\coloneqq\gamma(x-x^+),
\end{equation}
which is often called the gradient mapping \cite{nesterov2018lectures,beck2017first}.
The first-order optimality condition of \eqref{eq:prox-grad-map} leads
\begin{equation}
    \nabla f(x^+)-\nabla f(x)+\gamma(x-x^+)\in\widehat{\partial}F(x^+).
\end{equation}
Thus, $R_\gamma(x)=0$, equivalently, $x^+=x$ implies $0\in\widehat{\partial}F(x^+)\subset\partial F(x^+)$.
Accordingly, we employ $\|R_\gamma(x)\|$ as an optimality measure.

We also consider the following assumption for the composite problem \eqref{problem:composite}.

\begin{assumption}\label{assume:L-smoothness}
There exists $L>0$ such that
\begin{equation}\label{eq:descent-lemma}
    f(x)\le f(y)+\innerprod{\nabla f(y)}{x-y}+\frac{L}{2}\|x-y\|^2
\end{equation}
holds for any $x,y\in\dom g$.
\end{assumption}

The parameter $L$ is also called the upper curvature parameter.
It is well known that the descent lemma \eqref{eq:descent-lemma} is implied by the $L$-smoothness of $f$, namely,
\begin{equation}
    \|\nabla f(x)-\nabla f(y)\|\le L\|x-y\|
\end{equation}
holds for any $x,y\in\convhull(\dom g)$ \citep{nesterov2018lectures,beck2017first}.

The auto-conditioned proximal gradient method (AC-PGM) is summarized in Algorithm \ref{alg:AC-PGM}.

\begin{algorithm}[H]
\caption{Auto-conditioned proximal gradient method (AC-PGM)}
    \label{alg:AC-PGM}
    \begin{algorithmic}
    \STATE {\bfseries Input:} $x^0\in\dom g,~ \alpha>1,~ L_0>0$, and $k=1$.
    \REPEAT
    \STATE Compute
    \begin{align}
        \gamma_k &=\max\{L_0,\ldots,L_{k-1}\},\label{eq:stepsize-AC-PGM}\\
        x^k &\in\prox_{\frac{g}{\alpha\gamma_k}}\left(x^{k-1}-\frac{1}{\alpha\gamma_k}\nabla f(x^{k-1})\right),\label{eq:iterate-AC-PGM}\\
        L_k &=\frac{2(f(x^k)-f(x^{k-1})-\innerprod{\nabla f(x^{k-1})}{x^k-x^{k-1}})}{\|x^k-x^{k-1}\|^2}.\label{eq:estimate-AC-PGM}
    \end{align}
    \STATE Set $k\leftarrow k+1$.
    \UNTIL Termination criterion is satisfied.
    \end{algorithmic}
\end{algorithm}

Algorithm \ref{alg:AC-PGM} does not require any prior knowledge of the upper curvature parameter $L$ and any linesearch procedures.
If $x^k=x^{k-1}$, then $R_{\alpha\gamma_k}(x^{k-1})=0$; therefore, the algorithm can be terminated before computing $L_k$ in \eqref{eq:estimate-AC-PGM}.
Otherwise, $L_k \in \RR$ is well-defined because of $x^k\neq x^{k-1}$.

By the definition of $\gamma_k$ in \eqref{eq:stepsize-AC-PGM}, $\{\gamma_k\}$ is monotonically nondecreasing.
We introduce the index sets
\begin{equation}\label{eq:S}
    \SS\coloneqq\left\{k\ge1~\middle|~\beta\gamma_k\ge L_k\right\}, \text{ where }
    \beta \coloneqq \frac{\alpha+1}{2}>1;
    \qquad 
    \overline{\SS}\coloneqq\{1,2,\ldots\}\setminus\SS.
\end{equation}
$k\in\overline{\SS}$ means that the estimation of $L$ is not successful at the $k$-th iteration.

The following lemma is essential for our convergence analysis of Algorithm \ref{alg:AC-PGM}, which is valid without Assumption~\ref{assume:L-smoothness}.

\begin{lemma}\label{lem:gen-bound-AC-PGM}
Let $\{x^k\}$ be a sequence generated by Algorithm \ref{alg:AC-PGM} satisfying $x^k\neq x^{k-1}$ for all $k\ge1$.
Suppose that Assumption \ref{assume:AC-PGM} holds.
Then, we have for any $k \geq 1$ that
\begin{equation}\label{eq:gen-bound-AC-PGM}
    \frac{\alpha-1}{4\alpha^2}\sum_{l \in [k]} \frac{1}{\gamma_l}\|R_{\alpha \gamma_l}(x^{l-1})\|^2\le F(x^0)-F(x^k) + \sum_{l \in [k]\cap\overline{\SS}} \frac{\gamma_{l+1}-\gamma_l}{2}\|x^l-x^{l-1}\|^2.
\end{equation}
\end{lemma}

\begin{proof}
We obtain from the optimality of $x^k$ in \eqref{eq:iterate-AC-PGM} that
\begin{equation}
    \innerprod{\nabla f(x^{k-1})}{x^k-x^{k-1}}+\frac{\alpha\gamma_k}{2}\|x^k-x^{k-1}\|^2+g(x^k)-g(x^{k-1})\le0.
\end{equation}
Combining this with \eqref{eq:estimate-AC-PGM} yields
\begin{equation}\label{eq:pseudo-decrease-AC-PGM}
    \frac{\alpha\gamma_k-L_k}{2}\|x^k-x^{k-1}\|^2+F(x^k)-F(x^{k-1})\le0.
\end{equation}
If $k\in\SS$, by \eqref{eq:pseudo-decrease-AC-PGM} and the definition of $\SS$, we have 
\begin{align}
    F(x^{k-1})-F(x^k)\stackrel{\eqref{eq:pseudo-decrease-AC-PGM}}{\ge}\frac{\alpha\gamma_k-L_k}{2}\|x^k-x^{k-1}\|^2 &\ge\frac{\alpha-1}{4}\gamma_k\|x^k-x^{k-1}\|^2 \qquad (\because \beta\gamma_k-L_k\ge0)\label{eq:decrease-AC-PGM}\\
    &=\frac{\alpha-1}{4\alpha^2\gamma_k}\|R_{\alpha\gamma_k}(x^{k-1})\|^2.
\end{align}
On the other hand, $k\notin\SS$ implies $\gamma_{k+1}=\max\{L_0,\ldots,L_k\}=L_k$. 
Thus, it follows from \eqref{eq:pseudo-decrease-AC-PGM} that
\begin{align}
    \frac{\alpha-1}{2\alpha^2\gamma_k}\|R_{\alpha\gamma_k}(x^{k-1})\|^2
    &=\frac{\alpha-1}{2}\gamma_k\|x^k-x^{k-1}\|^2\\
    &\stackrel{\eqref{eq:pseudo-decrease-AC-PGM}}{\le} F(x^{k-1})-F(x^k)+\frac{\gamma_{k+1}-\gamma_k}{2}\|x^k-x^{k-1}\|^2.
\end{align}
By summing up, we conclude
\begin{equation}
    \frac{\alpha-1}{4\alpha^2}\sum_{l\in[k]}\frac{1}{\gamma_l}\|R_{\alpha\gamma_l}(x^{l-1})\|^2
    \le F(x^0)-F(x^k)+\sum_{l\in[k]\cap\overline{\SS}}\frac{\gamma_{l+1}-\gamma_l}{2}\|x^l-x^{l-1}\|^2.
\end{equation}
\end{proof}

Under Assumption \ref{assume:L-smoothness}, it holds that $L_k\le L$ and so $\gamma_k\le\max\{L_0,L\}$ for all $k\ge1$.
If $k\in\overline{\SS}$, then we have $\beta\gamma_k<L_k=\gamma_{k+1}\le\max\{L_0,L\}$, and hence
\begin{equation}\label{eq:failure-bound-AC-PGM}
    |\overline{\SS}|\le\left\lceil\log_\beta\frac{\max\{L_0,L\}}{L_0}\right\rceil.
\end{equation}
In other words, the estimation of $L$ fails at most finitely many times.

The convergence of Algorithm \ref{alg:AC-PGM} under Assumption \ref{assume:L-smoothness} is obtained as follows.

\begin{theorem}\label{thm:convergence-AC-PGM}
Let $\{x^k\}$ be a sequence generated by Algorithm \ref{alg:AC-PGM} satisfying $x^k\neq x^{k-1}$ for all $k\ge1$.
Suppose that Assumptions \ref{assume:AC-PGM} and \ref{assume:L-smoothness} hold.
Then the following assertions hold.
\begin{enumerate}[(i)]
    \item It follows that
    \begin{equation}\label{eq:sublinear-rate-AC-PGM}
        \min_{1\le l\le k}\|R_{\alpha\gamma_l}(x^{l-1})\| \le \sqrt{\frac{2\alpha^2\max\{L_0,L\}(2\Delta+C)}{(\alpha-1)k}}=\OO\left(k^{-\frac{1}{2}}\right)
    \end{equation}
    for all $k\ge1$, where $\Delta\coloneqq F(x^0)-F^*$ and 
    \begin{equation}\label{eq:constC-ACPGM}
        C \coloneqq (\max\{L_0,L\}-L_0)\max_{l \in \overline{\SS}}\|x^l-x^{l-1}\|^2 < \infty.
    \end{equation}
    \item The sequence $\{F(x^k)\}$ converges to a certain finite value and any accumulation point of $\{x^k\}$ is an l-stationary point of \eqref{problem:composite}.
\end{enumerate}
\end{theorem}

\begin{proof}
(i) By Lemma \ref{lem:gen-bound-AC-PGM}, we have
\begin{align}
    \frac{\alpha-1}{4\alpha^2\max\{L_0,L\}}k\min_{1\le l\le k}\|R_{\alpha\gamma_l}(x^{l-1})\|^2 &\le \frac{\alpha-1}{4\alpha^2}\sum_{l\in[k]}\frac{1}{\gamma_l}\|R_{\alpha\gamma_l}(x^{l-1})\|^2 \qquad (\because \gamma_l \leq \max\{L_0,L\})\\
    &\le F(x^0)-F(x^k)+\sum_{l\in[k]\cap\overline{\SS}}\frac{\gamma_{l+1}-\gamma_l}{2}\|x^l-x^{l-1}\|^2 \qquad (\because ~\text{Lemma~\ref{lem:gen-bound-AC-PGM}})\\
    &\le \Delta+\sum_{l\in[k]\cap\overline{\SS}}\frac{\gamma_{l+1}-\gamma_l}{2}\|x^l-x^{l-1}\|^2 \qquad (\because F(x^k)\ge F^*).
\end{align}
Rearranging this yields
\begin{equation}
    \min_{1\le l\le k}\|R_{\alpha\gamma_l}(x^{l-1})\| \le \sqrt{\frac{2\alpha^2\max\{L_0,L\}\{2\Delta+\sum_{l\in[k]\cap\overline{\SS}}(\gamma_{l+1}-\gamma_l)\|x^l-x^{l-1}\|^2\}}
    {(\alpha-1)k}}.
\end{equation}
It remains to show $\sum_{l\in[k]\cap\overline{\SS}}(\gamma_{l+1}-\gamma_l)\|x^l-x^{l-1}\|^2 \leq C$.
In fact, we have
\begin{align}
    &\sum_{l\in[k]\cap\overline{\SS}}(\gamma_{l+1}-\gamma_l)\|x^l-x^{l-1}\|^2
    \le\max_{l \in \overline{\SS}}\|x^l-x^{l-1}\|^2\sum_{l\in[k]\cap\overline{\SS}}(\gamma_{l+1}-\gamma_l)\\
    &\le \max_{l \in \overline{\SS}}\|x^l-x^{l-1}\|^2\sum_{l\in[k]}(\gamma_{l+1}-\gamma_l) \qquad (\because \text{the monotonicity of }\{\gamma_l\})\\
    &=(\gamma_{k+1}-\gamma_1)\max_{l \in \overline{\SS}}\|x^l-x^{l-1}\|^2\\
    &\le(\max\{L_0,L\}-L_0)\max_{l \in \overline{\SS}}\|x^l-x^{l-1}\|^2 \qquad (\because \gamma_{k+1}\le\max\{L_0,L\}\text{ and }\gamma_1=L_0).\\
    &=C
\end{align}
Since $\overline{\SS}$ is a finite set (see \eqref{eq:failure-bound-AC-PGM}), the constant $C$ is finite.

(ii) By \eqref{eq:decrease-AC-PGM} and the finiteness of $\overline{\SS}$, it holds that
\begin{equation}\label{eq:sufficient-decrease-AC-PGM}
    \frac{\alpha-1}{4}L_0\|x^k-x^{k-1}\|^2\le F(x^{k-1})-F(x^k)
\end{equation}
for all sufficiently large $k$.
Thus, we see from the boundedness from below of $F$ that $\{F(x^k)\}$ converges, and hence
\begin{equation}\label{eq:diff-vanish-AC-PGM}
    \|x^k-x^{k-1}\|\to0.
\end{equation}
Let $\{x^k\}_K$ be a subsequence of $\{x^k\}$ converging to some point $x^*$.
Then, $\{x^{k-1}\}_K$ also converges to $x^*$.
Since $x^k$ is optimal to the subproblem in \eqref{eq:iterate-AC-PGM}, we have
\begin{equation}
    \innerprod{\nabla f(x^{k-1})}{x^k-x^*}+\frac{\alpha\gamma_k}{2}\|x^k-x^{k-1}\|^2+g(x^k)\le \frac{\alpha\gamma_k}{2}\|x^*-x^{k-1}\|^2+g(x^*).
\end{equation}
By \eqref{eq:diff-vanish-AC-PGM} and the boundedness of $\{\gamma_k\}$, taking the upper limit $k\to_K\infty$ gives
\begin{equation}
    \limsup_{k\to_K\infty}g(x^k)\le g(x^*).
\end{equation}
Combining this with the lower semicontinuity of $g$ and continuity of $f$ yields $F(x^k)\to_KF(x^*)$.
As $\{F(x^k)\}$ converges, we have $\lim_{k\to\infty}F(x^k)=F(x^*)$, and hence $x^*\in\dom F=\dom g$.
From the optimality of $x^k$ in \eqref{eq:iterate-AC-PGM}, we have
\begin{equation}
    0\in\nabla f(x^{k-1})+\alpha\gamma_k(x^k-x^{k-1})+\widehat{\partial}g(x^k),
\end{equation}
which implies
\begin{equation}
    \xi^k \coloneqq \nabla f(x^k)-\nabla f(x^{k-1})+\alpha\gamma_k(x^{k-1}-x^k) \in \nabla f(x^k) + \widehat{\partial} g(x^k) = \widehat{\partial} F(x^k).
\end{equation}
We see from $\gamma_k\|x^k-x^{k-1}\|\to0$ and the continuity of $\nabla f$ that $\xi^k\to_K0$, which implies that $0\in\partial F(x^*)$.
\end{proof}

\begin{remark}
If $\alpha=1$, then the AC-PGM coincides with that of \citet{lan2024projected}.
If the convexity of $g$ is assumed, one can obtain
\begin{equation}
    \frac{2\alpha\gamma_k-L_k}{2}\|x^k-x^{k-1}\|^2+F(x^k)-F(x^{k-1})\le0
\end{equation}
instead of \eqref{eq:pseudo-decrease-AC-PGM}, and thus one may set $\alpha=1$ (more generally, $\alpha>1/2$).
On the other hand, unlike \citet{lan2024projected}, since we do not assume the convexity, it is necessary to set $\alpha>1$.
\end{remark}

From Theorem \ref{thm:convergence-AC-PGM}, we have the following complexity bound.

\begin{corollary}\label{cor:complexity-AC-PGM}
Under the same assumptions as in Theorem \ref{thm:convergence-AC-PGM},
denote $D\coloneqq \max_{l \in \overline{\SS}}\|x^l-x^{l-1}\|<\infty$.
Then, Algorithm \ref{alg:AC-PGM} finds an $\varepsilon$-stationary point satisfying $\|R_{\alpha\gamma_k}(x^{k-1})\|\le\varepsilon$ within
\begin{equation}\label{eq:compl-ACPGM-bounded}
    \frac{2\alpha^2\max\{L_0,L\}\{2\Delta+(\max\{L_0,L\}-L_0)D^2\}}{(\alpha-1)\varepsilon^2}
\end{equation}
iterations.
\end{corollary}

When $L_0\ge L$, all the iterations of the AC-PGM are successful and so it is merely the proximal gradient method with the constant stepsize $1/(\alpha L_0)$. In this case, the rate $\OO(\sqrt{L_0\Delta/k})$ of convergence given by Theorem~\ref{thm:convergence-AC-PGM} is well-known (see, e.g., \cite{beck2017first,nesterov2018lectures}).
When $L_0<L$, the iteration complexity \eqref{eq:compl-ACPGM-bounded} is of the form
\begin{equation}\label{eq:compl-order-ACPGM-diam}
    \OO\left(\frac{L\Delta}{\varepsilon^2} + \frac{L^2D^2}{\varepsilon^2}\right).
\end{equation}
Except the second term coming from unsuccessful iterations, the first term $\OO(L\Delta/\varepsilon^2)$ coincides with the lower complexity bound for smooth nonconvex optimization \cite{carmon2020lower}.

The AC-PGM is closely related to the auto-conditioned projected gradient method proposed by \citet[Algorithm 1]{lan2024projected}. Their method was analyzed when $g$ is convex with bounded domain, $f$ is $L$-smooth and $l$-weakly convex\footnote{Namely, $f(x)\geq f(y)+\innerprod{\nabla f(y)}{x-y}-\frac{l}{2}\|x-y\|^2$ for any $x,y \in \dom g$} on $\dom g$. It ensures the iteration complexity
\begin{equation}\label{eq:bound-lan2024}
\OO\left(\frac{LD_g}{\varepsilon} + \frac{LlD_g^2}{\varepsilon^2} + \log \frac{L}{L_0}\right) \qquad (\text{where }D_g\coloneqq \sup\{\|x-y\|:x,y \in \dom g\}),
\end{equation}
which interpolates the convergence rate between the convex and the weakly convex cases.
We remark that our analysis to obtain the complexity bound \eqref{eq:compl-order-ACPGM-diam} assumes neither the convexity of $g$, the boundedness of $\dom g$, nor the weak convexity of $f$.
It should also be noted that, since \cite{lan2024projected} conduct a unified analysis of both convex and nonconvex problems, their analysis becomes more complicated, whereas ours is simpler.

The auto-conditioned proximal gradient method by \citet{Hoai2025proximal} guarantees the rate $\OO(k^{-1/2})$ of convergence. However, it is stated for $k\geq \bar{k}$ with unknown index $\bar{k}$ (see \cite[Theorem 4.1]{Hoai2025proximal}). Moreover, \cite{Hoai2025proximal} imposes the $L$-smoothness of $f$, the convexity of $g$, and the quasiconvexity of the univariate function $t \mapsto \innerprod{\nabla f(x+t(y-x))}{y-x}$ on $[0,1]$ that are not assumed in our result.
They estimate the Lipschitz constant based on $\frac{\|\nabla f(x^k)-\nabla f(x^{k-1})\|}{\|x^k-x^{k-1}\|}$.
The convexity of $g$ and the quasiconvexity of the univariate function are employed to derive the descent property from the estimated Lipschitz constant.
In contrast, since we use \eqref{eq:estimate-AC-PGM}, such assumptions are not required.

\subsection{Convergence result under KL assumption}
The Kurdyka--\L ojasiewicz (KL) property is often used in the analysis of first-order methods to provide the convergence of the entire sequence and the convergence rate \cite{attouch2009convergence,attouch2010proximal,attouch2013convergence,bolte2014proximal,frankel2015splitting,jia2023convergence}, and it is also used for this purpose in this paper.
In this subsection, we assume the KL property of $F$.

\begin{assumption}\label{assume:KL-property}
For any $x^*\in\dom\partial F$, the objective function $F$ has the KL property at $x^*$, that is, there exists a positive constant $\varpi$, a neighborhood $\mathcal{U}$ of $x^*$, and a continuous concave function $\chi:[0,\varpi)\to[0,\infty)$ that is continuously differentiable on $(0,\varpi)$ and satisfies $\chi(0)=0$ as well as $\chi'(t)>0$ on $(0,\varpi)$, such that
\begin{equation}
    \chi'(F(x)-F(x^*))~\dist(0,\partial F(x))\ge1
\end{equation}
holds for all $x\in\mathcal{U}$ satisfying $F(x^*)<F(x)<F(x^*)+\varpi$.
\end{assumption}

If Assumption \ref{assume:KL-property} holds with  $\chi(t)=ct^{1-\theta}$ for some $c>0$ and $\theta\in[0,1)$, then we say that $F$ has the KL property of exponent $\theta$ at $x^*$.
It is known that wide classes of functions, including semialgebraic or subanalytic ones, admit the KL property (see, e.g., \cite{shiota1997geometry,bolte2007lojasiewicz,li2017calculus} and references therein).

We now provide convergence result for the AC-PGM in the presence of the KL property.

Let $\{x^k\}_K$ be a subsequence of $\{x^k\}$ converging $x^*$.
In view of Theorem \ref{thm:convergence-AC-PGM}, we have $0\in\partial F(x^*)$, and hence it holds that $x^*\in\dom\partial F$.
Recalling the proof of Theorem \ref{thm:convergence-AC-PGM}, we see that
\begin{align}
    &\lim_{k\to\infty}F(x^k)=F(x^*),\\
    &\lim_{k\to\infty}\|x^k-x^{k-1}\|=0,\\
    &F(x^k)\le F(x^{k-1})-\frac{\alpha-1}{4}L_0\|x^k-x^{k-1}\|^2,\\
    &\xi^k=\nabla f(x^k)-\nabla f(x^{k-1})+\alpha\gamma_k(x^{k-1}-x^k) \in\partial F(x^k),
\end{align}
hold for all sufficiently large $k$.
By assuming the $L$-smoothness of $f$, which is stronger than Assumption \ref{assume:L-smoothness}, we obtain
\begin{align}
    \|\xi^k\| &=\|\nabla f(x^k)-\nabla f(x^{k-1})+\alpha\gamma_k(x^{k-1}-x^k)\|\\
    &\le(L+\alpha\max\{L_0,L\})\|x^{k-1}-x^k\|
\end{align}
because $\gamma_k\le\max\{L_0,L\}$ for all $k\ge1$.
Applying the convergence results for abstract descent methods \cite{frankel2015splitting,qian2025superlinear} yields the following result.

\begin{theorem}\label{thm:KL}
Let $\{x^k\}$ be a sequence generated by Algorithm \ref{alg:AC-PGM} satisfying $x^k\neq x^{k-1}$ for all $k\ge1$.
Suppose that Assumptions \ref{assume:AC-PGM} and \ref{assume:KL-property} hold, $f$ in $L$-smooth, and there exists a subsequence of $\{x^k\}$ converging $x^*$.
Then $\sum_{k=0}^\infty\|x^k-x^{k-1}\|<\infty$ and $F(x^k)\to F(x^*)$ hold, particularly, $\{x^k\}$ also converges to $x^*$.
Moreover, $F$ has the KL property of exponent $\theta\in(0,1)$, then the following assertions hold:
\begin{enumerate}[(i)]
    \item If $\theta\in(0,1/2)$, then $\{F(x^k)\}$ and $\{x^k\}$ converges Q-superlinearly of order $\frac{1}{2\theta}$;
    \item If $\theta=1/2$, then $\{F(x^k)\}$ and $\{x^k\}$ converges Q-linearly and R-linearly, respectively;
    \item If $\theta\in(1/2,1)$, then there exist $c_1, c_2>0$ such that
    \begin{align}
        F(x^k)-F(x^*)\le c_1k^{-\frac{1}{2\theta-1}},\\
        \|x^k-x^*\|\le c_2k^{-\frac{1-\theta}{2\theta-1}}.
    \end{align}
\end{enumerate}
\end{theorem}

We note that the superlinear convergence result for lower exponents ($\theta\in(0,1/2)$) have recently appeared in \cite{qian2025superlinear,bento2025convergence,yagishita2025proximal}.

\begin{remark}\label{rem:KL}
Since we obtain from \eqref{eq:decrease-AC-PGM} and the finiteness of $\overline{\SS}$ that
\begin{align}
    \frac{\alpha-1}{4\alpha^2\max\{L_0,L\}}\|R_{\alpha\gamma_k}(x^{k-1})\|^2 &\le F(x^{k-1})-F(x^k)\\
    &\le F(x^{k-1})-F(x^*) \qquad (\because F(x^k)\ge F(x^*)).
\end{align}
for all sufficiently large $k$, convergence rates can also be obtained for the optimality measure.
For example, if $\theta=1/2$, the convergence rate of $\|R_{\alpha\gamma_k}(x^{k-1})\|$ is also linear.
\end{remark}

\subsection{Adaptivity to weak smoothness}
We shall show that Algorithm~\ref{alg:AC-PGM} is adaptive not only to the upper curvature parameter but also to the weak smoothness.
We consider the following weak smoothness assumption.

\begin{assumption}\label{assume:weak-smooth}
    There exist $\nu \in (0,1)$ and $L_\nu >0$ such that
    \begin{equation}\label{eq:wsm}
       f(x) \leq f(y) + \innerprod{\nabla f(y)}{y-x} + \frac{L_\nu}{1+\nu} \|x-y\|^{1+\nu}
    \end{equation}
    holds for any $x,y \in \dom g$.
\end{assumption}

Similar to the descent lemma \eqref{eq:descent-lemma}, a sufficient condition for this assumption is a H\"older continuity of $\nabla f$ on $\convhull(\dom g)$, that is,
\begin{equation}\label{eq:Holder-cont}
    \|\nabla f(x) - \nabla f(y)\| \le L_\nu \|x-y\|^\nu,\quad \forall x,y \in \convhull(\dom g).
\end{equation}
We will use Lemma \ref{lem:gen-bound-AC-PGM} for the convergence analysis under the weak smoothness.
In contrast to the setting where Assumption \ref{assume:L-smoothness} holds, the index set $\overline{\SS}$ of unsuccessful iterations may not be finite in the weakly smooth case.
Nevertheless, the next fact shows that the accumulation term in Lemma~\ref{lem:gen-bound-AC-PGM} can be bounded by a constant.
\begin{lemma}\label{lem:wsm-AC-PGM}
Let $\{x^k\}$ be a sequence generated by Algorithm \ref{alg:AC-PGM} satisfying $x^k\neq x^{k-1}$ for all $k\ge1$.
Suppose that Assumptions~\ref{assume:AC-PGM} and \ref{assume:weak-smooth} hold.
Then, for any $k\geq 1$, we have
\begin{equation}\label{eq:wsm-ACPGM-bound-err}
    \sum_{l \in [k]\cap \overline{\SS}} \frac{\gamma_{l+1}-\gamma_l}{2}\|x^l-x^{l-1}\|^2 \leq C_\nu \coloneqq 
    \frac{1}{1-\beta^{-\frac{1+\nu}{1-\nu}}}\left(\frac{L_\nu}{1+\nu}\right)^{\frac{2}{1-\nu}}\left(\frac{2}{L_0}\right)^{\frac{1+\nu}{1-\nu}}.
\end{equation}
\end{lemma}

\begin{proof}
Define the H\"older coefficient estimate $L_{\nu,k}$ so that
$$
    f(x^k) - f(x^{k-1}) - \innerprod{\nabla f(x^{k-1})}{x^k-x^{k-1}} = \frac{L_{\nu,k}}{1+\nu} \|x^k-x^{k-1}\|^{1+\nu}.
$$
The constant $L_{\nu,k} \in \RR$ is well-defined since $x^k \ne x^{k-1}$.
Clearly, $L_{\nu,k}\leq L_\nu$ holds.
Moreover, we have the expression of $L_k$ in \eqref{eq:estimate-AC-PGM} using $\nu$ and $L_{\nu,k}$ as follows.
\begin{equation}\label{eq:L_k-univ}
    L_k = \frac{2}{1+\nu}L_{\nu,k}\frac{1}{\|x^k-x^{k-1}\|^{1-\nu}} = \frac{2}{1+\nu}\frac{L_{\nu,k}\alpha^{1-\nu}\gamma_k^{1-\nu}}{\|R_{\alpha \gamma_k}(x^{k-1})\|^{1-\nu}}.
\end{equation}
If $k \in \overline{\mathcal{S}}$ then $\beta\gamma_k < L_k = \gamma_{k+1}$ and
\begin{equation}
    0\le \gamma_{k+1}-\gamma_k = L_k-\gamma_k \stackrel{\eqref{eq:L_k-univ}}{=} \gamma_k^{1-\nu}\left(\frac{2}{1+\nu}\frac{L_{\nu,k}\alpha^{1-\nu}}{\|R_{\alpha \gamma_k}(x^{k-1})\|^{1-\nu}} - \gamma_k^\nu \right).
\end{equation}
Rearranging this inequality and using $L_{\nu,k}\leq L_\nu$, we obtain the following bound on the residue.
\begin{equation}\label{eq:wsm-residue-bound}
    \|R_{\alpha\gamma_k}(x^{k-1})\| \leq \left(\frac{2}{1+\nu}\frac{L_{\nu}\alpha^{1-\nu}}{\gamma_k^\nu}\right)^{\frac{1}{1-\nu}},\quad \forall k \in \overline{\SS}.    
\end{equation}
Therefore, it follows that
\begin{align}
\sum_{l \in [k]\cap\overline{\mathcal{S}}} \frac{\gamma_{l+1}-\gamma_l}{2}\|x^l-x^{l-1}\|^2 &= \sum_{l \in [k]\cap\overline{\mathcal{S}}} \frac{L_l-\gamma_l}{2}\frac{\|R_{\alpha\gamma_l}(x^{l-1})\|^2}{\alpha^2\gamma_l^2}
\leq 
\sum_{l \in [k]\cap\overline{\mathcal{S}}} \frac{L_l}{2}\frac{\|R_{\alpha\gamma_l}(x^{l-1})\|^2}{\alpha^2\gamma_l^2} \nonumber\\
&\stackrel{\eqref{eq:L_k-univ}}{=}
\sum_{l \in [k]\cap\overline{\mathcal{S}}} \frac{\|R_{\alpha\gamma_l}(x^{l-1})\|^2}{2\alpha^2\gamma_l^2} \frac{2}{1+\nu}\frac{L_{\nu,l}\alpha^{1-\nu}\gamma_l^{1-\nu}}{\|R_{\alpha\gamma_l}(x^{l-1})\|^{1-\nu}} \nonumber\\
&=
\frac{1}{(1+\nu)\alpha^{1+\nu}}
\sum_{l \in [k]\cap\overline{\mathcal{S}}} \frac{L_{\nu,l}\|R_{\alpha\gamma_l}(x^{l-1})\|^{1+\nu}}{\gamma_l^{1+\nu}} \nonumber\\
&\leq
2^{\frac{1+\nu}{1-\nu}} \left(\frac{L_\nu}{1+\nu}\right)^{\frac{2}{1-\nu}}
\sum_{l \in [k]\cap\overline{\mathcal{S}}} 
\frac{1}{\gamma_l^{\frac{1+\nu}{1-\nu}}}
\quad (\because \eqref{eq:wsm-residue-bound} \text{ and } L_{\nu,l}\leq L_\nu) 
\label{eq:wsm-errbd}
\end{align}
It remains to estimate $\sum_{l \in [k]\cap\overline{\mathcal{S}}} \gamma_l^{-\frac{1+\nu}{1-\nu}}$.
Denote $\mu=\frac{1+\nu}{1-\nu}$ and $[k]\cap \overline{\mathcal{S}} = \{k_1,k_2,\ldots,k_s\}$ with $k_j<k_{j+1}$.
By the definition of $\overline{\SS}$, it follows that
$$\beta\gamma_{k_j} < L_{k_j} = \gamma_{1+k_j} \leq \gamma_{k_{j+1}},\quad j=1,\ldots,s-1.$$
Therefore, we obtain
$$
\sum_{l \in [k]\cap\overline{\mathcal{S}}} \frac{1}{\gamma_l^{\mu}}
=
\sum_{j=1}^s \frac{1}{\gamma_{k_j}^{\mu}}
\leq
\sum_{j=1}^{s}\frac{1}{(\beta^{j-1} \gamma_{k_1})^{\mu}} 
\leq
\frac{1}{L_0^{\mu}}\sum_{j=1}^{s}\frac{1}{(\beta^{\mu})^{j-1}}
\leq \frac{1}{L_0^{\mu}} \frac{1}{1-\beta^{-\mu}}.
$$
The assertion follows by combining this and \eqref{eq:wsm-errbd}.
\end{proof}

Now we are ready to establish the convergence result under the weak smoothness.

\begin{theorem}\label{th:conve-rate-weak-sm}
Let $\{x^k\}$ be a sequence generated by Algorithm~\ref{alg:AC-PGM} satisfying $x^k\ne x^{k-1}$ for all $k\geq 1$.
Suppose that Assumptions~\ref{assume:AC-PGM} and \ref{assume:weak-smooth} hold.
Then, for any $k\geq 1$, we have
$$
\min_{l \in [k]} \|R_{\alpha \gamma_l}(x^{l-1})\| \leq
\alpha \cdot \max\left\{ 2\sqrt{\frac{ L_0(\Delta+C_\nu)}{(\alpha-1)k}},~~ 
\left(\frac{2^{1+2\nu}L_\nu}{1+\nu}\right)^{\frac{1}{1+\nu}}\left(\frac{\Delta+C_\nu}{(\alpha-1)k}\right)^{\frac{\nu}{1+\nu}} \right\} = \OO(k^{-\frac{\nu}{1+\nu}}),
$$
where $C_\nu$ is the constant defined in \eqref{eq:wsm-ACPGM-bound-err}.
\end{theorem}

\begin{proof}
For $k\geq 1$, we see that
\begin{align*}
\gamma_k &\leq \gamma_{k+1} = \max\{L_0,L_1,\ldots,L_{k}\}
\stackrel{\eqref{eq:L_k-univ}}{=} \max\left\{L_0,\max_{l \in [k]} \frac{2}{1+\nu}\frac{L_{\nu,l}\alpha^{1-\nu}\gamma_l^{1-\nu}}{\|R_{\alpha \gamma_l}(x^{l-1})\|^{1-\nu}}\right\}\\
&\leq \max\left\{L_0,\left(\max_{l \in [k]} \frac{2}{1+\nu}\frac{L_{\nu,l}\alpha^{1-\nu}}{\|R_{\alpha \gamma_l}(x^{l-1})\|^{1-\nu}}\right) \gamma_{k}^{1-\nu}  \right\} \qquad (\because \text{the monotonicity of }\{\gamma_l\})\\
& \leq \max\left\{L_0^\nu,
 \frac{2}{1+\nu}\frac{L_{\nu}\alpha^{1-\nu}}{\min_{l \in [k]}\|R_{\alpha \gamma_l}(x^{l-1})\|^{1-\nu}} \right\}\gamma_{k}^{1-\nu} \qquad (\because  L_0 = L_0^\nu L_0^{1-\nu}\leq L_0^\nu\gamma_k^{1-\nu}  \text{ and }  L_{\nu,l} \leq L_\nu)  
\end{align*}
Therefore, the following upper bound on $\gamma_k$ is obtained for all $k\geq 1$.
$$
\gamma_k \leq \max\left\{L_0, \left(\frac{2L_\nu\alpha^{1-\nu}}{1+\nu}\right)^{\frac{1}{\nu}} \frac{1}{\min_{l \in [k]} \|R_{\alpha \gamma_l}(x^{l-1})\|^{\frac{1-\nu}{\nu}} }\right\}.
$$
Using this, it follows that
\begin{equation}\label{eq:wsm-resdue-lb}
\frac{1}{\gamma_k} \min_{l \in [k]}\|R_{\alpha \gamma_l}(x^{l-1})\|^2
\geq
\min\left\{
	\frac{\min_{l \in [k]}\|R_{\alpha \gamma_l}(x^{l-1})\|^2}{L_0}, \left(\frac{2L_\nu\alpha^{1-\nu}}{1+\nu}\right)^{-\frac{1}{\nu}} \min_{l \in [k]}\|R_{\alpha \gamma_l}(x^{l-1})\|^{\frac{1+\nu}{\nu}}
\right\}.
\end{equation}
Combining Lemmas \ref{lem:gen-bound-AC-PGM}, \ref{lem:wsm-AC-PGM}, and \eqref{eq:wsm-resdue-lb}, we conclude that
$$
\min_{l \in [k]}\|R_{\alpha \gamma_l}(x^{l-1})\| \le \max\left\{ \sqrt{\frac{4\alpha ^2 L_0(\Delta+C_\nu)}{(\alpha-1)k}},~~ \left(\frac{4\alpha^2\left(\frac{2L_\nu\alpha^{1-\nu}}{1+\nu}\right)^{1/\nu}(\Delta+C_\nu)}{(\alpha-1)k}\right)^{\frac{\nu}{1+\nu}} \right\}.
$$
This completes the proof.
\end{proof}

Theorem \ref{th:conve-rate-weak-sm} is a ``universal'' result in the sense that it is adaptive to every acceptable H\"older exponent $\nu \in (0,1]$.
In particular, the iteration complexity to achieve $\|R_{\alpha \gamma_k}(x^{k-1})\|\leq \varepsilon$ is bounded by
$$
\OO\left(\inf_{\nu \in (0,1)}\max\left\{
\frac{L_0}{\varepsilon^2},
\frac{L_\nu^\frac{1}{\nu}}{\varepsilon^{\frac{1+\nu}{\nu}}}
\right\}(\Delta+C_\nu)\right).
$$
This bound is guaranteed under Assumption \ref{assume:weak-smooth}, which is weaker than the H\"older continuity of $\nabla f$. Below, let $\nabla f$ satisfy the H\"older continuity \eqref{eq:Holder-cont}.
The universality behavior is similar to the linesearch-free method \cite{oikonomidis24adaptive} established in convex setting.
Excepting the constant $C_\nu$ corresponding to the accumulation term of unsuccessful iterations, the factor $\OO(L_\nu^{\frac{1}{\nu}}\Delta/\varepsilon^{\frac{1+\nu}{\nu}})$ matches the best-known iteration complexity bound guaranteed by the first-order methods in the presence of the weak smoothness \cite{yashtini2016holder,cartis17regularization,dvurechensky17inexact}. Note that the first-order methods in \cite{cartis17regularization,dvurechensky17inexact} are based on backtracking (or trust-region) strategy and possess the universality.

\subsection{Extension to Bregman proximal gradient method}
Let us consider extending the AC-PGM to the Bregman proximal gradient method for the composite problem \eqref{problem:composite}.
We first define the Bregman divergence.
Let $h:\EE\to(-\infty,\infty]$ be a lower semicontinuous strictly convex function being continuously differentiable on $\CC\coloneqq\interior\dom h$.
Then the Bregman divergence generated by the kernel $h$ is defined by
\begin{align}
    D_h(x,y)\coloneqq\begin{cases}
        h(x)-h(y)-\innerprod{\nabla h(y)}{x-y}, \quad&y\in\CC,\\
        \infty, &y\notin\CC.
    \end{cases} 
\end{align}
By the strict convexity of $h$, $D_h(x,y)=0$ if and only if $x=y$.

The the Bregman proximal gradient method iterates
\begin{equation}\label{eq:Bregman-prox-grad-map}
    x^+\in\argmin_{y\in\EE}\left\{\innerprod{\nabla f(x)}{y}+\gamma D_h(y,x)+g(y)\right\},
\end{equation}
where $\gamma>0$ and $x\in\CC\cap\dom g$.
When choosing $h(x)=\frac{1}{2}\|x\|^2$, the iteration \eqref{eq:Bregman-prox-grad-map} reduces to that of the standard proximal gradient method.
Assuming $x^+\in\CC$ and defining $R_\gamma(x)\coloneqq\gamma(x-x^+)$, as in the case of the proximal gradient method, $R_\gamma(x)=0$ implies the stationarity of $x^+$.
To ensure the well-definedness of the iterations of the Bregman proximal gradient method, we make the following assumption.

\begin{assumption}\label{assume:AC-BPGM}
For any $\gamma>0,~ x\in\CC\cap\dom g$,
\begin{equation}
    \argmin_{y\in\EE}\left\{\innerprod{\nabla f(x)}{y}+\gamma D_h(y,x)+g(y)\right\}
\end{equation}
is nonempty and included in $\CC\cap\dom g$.
\end{assumption}

We introduce the relative smoothness condition as follows.
This property is often assumed in the analysis of the Bregman proximal gradient method \citep{van2017forward,bauschke2017descent,lu2018relatively}, which is an extension of the descent lemma.

\begin{assumption}\label{assume:relative-L-smoothness}
There exists $L_h>0$ such that
\begin{equation}\label{eq:relative-descent-lemma}
    f(x)\le f(y)+\innerprod{\nabla f(y)}{x-y}+L_hD_h(x,y)
\end{equation}
holds for any $x,y\in\CC\cap\dom g$.
\end{assumption}

The function $f$ is said to be $L_h$-smooth relative to $h$ if Assumption \ref{assume:relative-L-smoothness} holds.
When $h(x)=\frac{1}{2}\|x\|^2$, Assumption \ref{assume:relative-L-smoothness} coincides with the standard descent lemma \eqref{eq:descent-lemma}.

The auto-conditioned Bregman proximal gradient method (AC-BPGM) is summarized in Algorithm \ref{alg:AC-BPGM}.

\begin{algorithm}[H]
\caption{Auto-conditioned Bregman proximal gradient method (AC-BPGM)}
    \label{alg:AC-BPGM}
    \begin{algorithmic}
    \STATE {\bfseries Input:} $x^0\in\CC\cap\dom g,~ \alpha>1,~ L_0>0$, and $k=1$.
    \REPEAT
    \STATE Compute
    \begin{align}
        \gamma_k &=\max\{L_0,\ldots,L_{k-1}\},\label{eq:stepsize-AC-BPGM}\\
        x^k &\in\argmin_{y\in\EE}\left\{\innerprod{\nabla f(x^{k-1})}{y}+\alpha\gamma_kD_h(y,x^{k-1})+g(y)\right\},\label{eq:iterate-AC-BPGM}\\
        L_k &=\frac{f(x^k)-f(x^{k-1})-\innerprod{\nabla f(x^{k-1})}{x^k-x^{k-1}}}{D_h(x^k,x^{k-1})}.\label{eq:estimate-AC-BPGM}
    \end{align}
    \STATE Set $k\leftarrow k+1$.
    \UNTIL Termination criterion is satisfied.
    \end{algorithmic}
\end{algorithm}

The AC-BPGM is the generalization of the standard AC-PGM in which the determination of $L_k$ is adapted to the relative smoothness.
As in the AC-PGM, if $x^k=x^{k-1}$, it means that $x^k$ is an l-stationary point; otherwise, $L_k$ is well-defined.

Setting $\beta=\frac{\alpha+1}{2}>1$ and defining $\SS$ in the same way as in the AC-PGM (see \eqref{eq:S}), Assumption \ref{assume:relative-L-smoothness} ensures that the following similar properties hold for the AC-BPGM as well:
\begin{align}
    L_k &\le L_h,\\
    L_0 &=\gamma_1\le\gamma_2\le\cdots\le\max\{L_0,L_h\},\\
    |\overline{\SS}| &\le\left\lceil\log_\beta\frac{\max\{L_0,L_h\}}{L_0}\right\rceil.\label{eq:failure-bound-AC-BPGM}
\end{align}

As in Lemma \ref{lem:gen-bound-AC-PGM}, we have the following lemma.
The proof is omitted since it is similar.

\begin{lemma}\label{lem:gen-bound-AC-BPGM}
Let $\{x^k\}$ be a sequence generated by Algorithm \ref{alg:AC-BPGM} satisfying $x^k\neq x^{k-1}$ for all $k\ge1$.
Suppose that Assumption \ref{assume:AC-BPGM} holds.
Then, we have for any $k \geq 1$ that
\begin{equation}\label{eq:gen-bound-AC-BPGM}
\frac{\alpha-1}{2}\sum_{l\in[k]}\gamma_lD_h(x^l,x^{l-1})\le F(x^0)-F(x^k) + \sum_{l\in[k]\cap\overline{\SS}}(\gamma_{l+1}-\gamma_l)D_h(x^l,x^{l-1}).
\end{equation}
\end{lemma}

Using Lemma \ref{lem:gen-bound-AC-BPGM}, we have the following convergence result of the AC-BPGM.

\begin{theorem}\label{thm:convergence-AC-BPGM}
Let $\{x^k\}$ be a sequence generated by Algorithm \ref{alg:AC-BPGM} satisfying $x^k\neq x^{k-1}$ for all $k\ge1$.
Suppose that Assumptions \ref{assume:AC-BPGM} and \ref{assume:relative-L-smoothness} hold.
Assume further that $h$ is $\sigma$-strongly convex, namely,
\begin{equation}
    D_h(x,y)\ge\frac{\sigma}{2}\|x-y\|^2
\end{equation}
holds for any $x\in\EE$ and $y\in\CC$, where $\sigma>0$.
Then the following assertions hold.
\begin{enumerate}[(i)]
    \item It holds that
    \begin{equation}\label{eq:sublinear-rate-AC-BPGM}
        \min_{1\le l\le k}\|R_{\alpha\gamma_l}(x^{l-1})\| \le \sqrt{\frac{4\alpha^2\max\{L_0,L_h\}(\Delta+C)}{(\alpha-1)\sigma k}}=\OO\left(k^{-\frac{1}{2}}\right)
    \end{equation}
    for all $k\ge1$, where $\Delta\coloneqq F(x^0)-F^*$ and
    \begin{equation}
        C \coloneqq (\max\{L_0,L_h\}-L_0)\max_{l \in \overline{\SS}}D_h(x^l,x^{l-1}) < \infty.
    \end{equation}
    \item The sequence $\{F(x^k)\}$ converges to a certain finite value.
    \item Any accumulation point of $\{x^k\}$ is an l-stationary point of \eqref{problem:composite} if $\CC=\EE$.
\end{enumerate}
\end{theorem}

\begin{proof}
(i) We see from Lemma \ref{lem:gen-bound-AC-BPGM} and the strong convexity of $h$ that
\begin{align}
    \frac{(\alpha-1)\sigma}{4\alpha^2\max\{L_0,L_h\}}k\min_{1\le l\le k}\|R_{\alpha\gamma_l}(x^{l-1})\|^2 &\le \frac{(\alpha-1)\sigma}{4\alpha^2}\sum_{l\in[k]}\frac{1}{\gamma_l}\|R_{\alpha\gamma_l}(x^{l-1})\|^2 \qquad (\because \gamma_l \leq \max\{L_0,L_h\})\\
    &=\frac{(\alpha-1)\sigma}{4}\sum_{l\in[k]}\gamma_l\|x^l-x^{l-1}\|^2\\
    &\le\frac{\alpha-1}{2}\sum_{l\in[k]}\gamma_lD_h(x^l,x^{l-1}) \qquad (\because \text{the strong convexity of }h)\\
    &\le F(x^0)-F(x^k)+\sum_{l\in[k]\cap\overline{\SS}}(\gamma_{l+1}-\gamma_l)D_h(x^l,x^{l-1}) \qquad (\because \text{Lemma \ref{lem:gen-bound-AC-BPGM}})\\
    &\le \Delta+\sum_{l\in[k]\cap\overline{\SS}}(\gamma_{l+1}-\gamma_l)D_h(x^l,x^{l-1}) \qquad (\because F(x^k)\ge F^*).
\end{align}
Rearranging this yields
\begin{equation}
    \min_{1\le l\le k}\|R_{\alpha\gamma_l}(x^{l-1})\| \le \sqrt{\frac{4\alpha^2\max\{L_0,L_h\}\{\Delta+\sum_{l\in[k]\cap\overline{\SS}}(\gamma_{l+1}-\gamma_l)D_h(x^l,x^{l-1})\}}{(\alpha-1)\sigma k}}
\end{equation}
The assertion (i) is obtained from this estimate combined with the following bound.
$$
\sum_{l\in[k]\cap\overline{\SS}}(\gamma_{l+1}-\gamma_l)D_h(x^l,x^{l-1}) \leq \max_{l\in\overline{\SS}}D_h(x^l,x^{l-1})\sum_{l=1}^k(\gamma_{l+1}-\gamma_l)
=
\max_{l\in\overline{\SS}}D_h(x^l,x^{l-1}) (\gamma_{k+1}-\gamma_1) \leq C,
$$
where the first inequality follows from the monotonicity of $\{\gamma_l\}$.

(ii) As in the proof of Theorem \ref{thm:convergence-AC-PGM}, the strong convexity of $h$ yields
\begin{equation}\label{eq:sufficient-decrease-AC-BPGM}
    \frac{(\alpha-1)\sigma}{4}L_0\|x^k-x^{k-1}\|^2\le F(x^{k-1})-F(x^k)
\end{equation}
for all sufficiently large $k$.
Thus, we see from the boundedness from below of $F$ that $\{F(x^k)\}$ converges, and hence
\begin{equation}\label{eq:diff-vanish-AC-BPGM}
    \|x^k-x^{k-1}\|\to0.
\end{equation}

(iii) Assume that $\CC=\EE$.
Let $\{x^k\}_K$ be a subsequence of $\{x^k\}$ converging to some point $x^*$.
Then, $\{x^{k-1}\}_K$ also converges to $x^*$.
Since $x^k$ is optimal to the subproblem in \eqref{eq:iterate-AC-BPGM}, we have
\begin{equation}
    \innerprod{\nabla f(x^{k-1})}{x^k-x^*}+\alpha\gamma_kD_h(x^k,x^{k-1})+g(x^k)\le \alpha\gamma_kD_h(x^*,x^{k-1})+g(x^*).
\end{equation}
Since $\{\gamma_k\}$ is bounded and both $\{x^k\}_K$ and $\{x^{k-1}\}_K$ converge to $x^*$, taking the upper limit $k\to_K\infty$ gives
\begin{equation}
    \limsup_{k\to_K\infty}g(x^k)\le g(x^*).
\end{equation}
Combining this with the lower semicontinuity of $g$ and continuity of $f$ yields $F(x^k)\to_KF(x^*)$.
As $\{F(x^k)\}$ converges, we have $\lim_{k\to\infty}F(x^k)=F(x^*)$, and hence $x^*\in\dom F=\dom g$.
From the optimality of $x^k$ in \eqref{eq:iterate-AC-BPGM}, we have
\begin{equation}
    0\in\nabla f(x^{k-1})+\alpha\gamma_k(\nabla h(x^k)-\nabla h(x^{k-1}))+\widehat{\partial}g(x^k),
\end{equation}
which implies
\begin{equation}
    \xi^k \coloneqq \nabla f(x^k)-\nabla f(x^{k-1})+\alpha\gamma_k(\nabla h(x^k)-\nabla h(x^{k-1})) \in \nabla f(x^k) + \widehat{\partial} g(x^k) = \widehat{\partial} F(x^k).
\end{equation}
We see from the boundedness of $\{\gamma_k\}$ and the continuity of $\nabla f$ and $\nabla h$ that $\xi^k\to_K0$, which implies that $0\in\partial F(x^*)$.
\end{proof}

Note that Theorem \ref{thm:convergence-AC-BPGM} is a generalization of Theorem \ref{thm:convergence-AC-PGM}.
In fact, Theorem  \ref{thm:convergence-AC-BPGM} reduces to Theorem  \ref{thm:convergence-AC-PGM} when $h(x)=\frac{1}{2}\|x\|^2$.

\section{Other linesearch-free first-order methods}\label{sec:other-FoM}
In this section, to demonstrate that the auto-conditioned stepsize is a general stepsize strategy, we propose two linesearch-free first-order methods other than the proximal gradient-type methods and conduct convergence analyses.
Specifically, (generalized) conditional gradient method and Riemannian gradient method are examined in Subsections \ref{subsec:CGM} and \ref{subsec:RGM}, respectively.
Although the analyses for each algorithm slightly differ from the AC-PGM, they share the same essential principle.

\subsection{Conditional gradient method}\label{subsec:CGM}
To consider the auto-conditioned conditional gradient method (AC-CGM) for the composite problem \eqref{problem:composite}, in addition to Assumption \ref{assume:L-smoothness}, the following assumptions are made.

\begin{assumption}\label{assume:AC-CGM}
~
\begin{enumerate}[(i)]
    \item $g$ is convex function with bounded domain, namely, $D_g\coloneqq\sup_{x,y\in\dom g}\|x-y\|<\infty$;
    \item For any $x\in\dom g$,
    \begin{equation}\label{eq:subproblem-CGM}
        \argmin_{v\in\mathbb{E}}\left\{\innerprod{\nabla f(x)}{v}+g(v)\right\}
    \end{equation}
    is nonempty.
\end{enumerate}
\end{assumption}

The Frank--Wolfe gap at $x\in\dom g$ is defined by
\begin{equation}
    G(x)\coloneqq\max_{v\in\mathbb{E}}\left\{\innerprod{\nabla f(x)}{x-v}+g(x)-g(v)\right\}\ge0.
\end{equation}
It is easy to see that $G(x)=\innerprod{\nabla f(x)}{x-v^*}+g(x)-g(v^*)$ where $v^*$ is any solution of the subproblem \eqref{eq:subproblem-CGM}.
Thus, the Frank--Wolfe gap is a computable quantity within the algorithm.
It is not hard to see that $G(x^*)=0$ if and only if $x^*$ is an l-stationary point of \eqref{problem:composite} \citep[Theorem 13.6]{beck2017first}.
Therefore, the Frank--Wolfe gap can be used as an optimality measure.
Moreover, $G$ is lower semicontinuous (see, e.g., \citep[Lemma 2.2]{yagishita2025convergence}).

The AC-CGM is summarized in Algorithm \ref{alg:AC-CGM}.

\begin{algorithm}[H]
\caption{Auto-conditioned conditional gradient method (AC-CGM)}
    \label{alg:AC-CGM}
    \begin{algorithmic}
    \STATE {\bfseries Input:} $x^0\in\dom g,~ \alpha>\frac{1}{2},~ L_0>0$, and $k=1$.
    \REPEAT
    \STATE Compute
    \begin{align}
        v^k&=\argmin_{v\in\mathbb{E}}\left\{\innerprod{\nabla f(x^{k-1})}{v}+g(v)\right\}\label{eq:subproblem-AC-CGM},\\
        \gamma_k &=\max\{L_0,\ldots,L_{k-1}\},\label{eq:stepsize-AC-CGM}\\
        G_k &=G(x^{k-1})=\innerprod{\nabla f(x^{k-1})}{x^{k-1}-v^k}+g(x^{k-1})-g(v^k)\label{eq:F-W-gap-AC-CGM},\\
        \tau_k &=\min\left\{1,\frac{G_k}{\alpha\gamma_k\|x^{k-1}-v^k\|^2}\right\},\label{eq:tau-AC-CGM}\\
        x^k &=(1-\tau_k)x^{k-1}+\tau_kv^k,\label{eq:iterate-AC-CGM}\\
        L_k &=\frac{2(f(x^k)-f(x^{k-1})-\innerprod{\nabla f(x^{k-1})}{x^k-x^{k-1}})}{\|x^k-x^{k-1}\|^2}.\label{eq:estimate-AC-CGM}
    \end{align}
    \STATE Set $k\leftarrow k+1$.
    \UNTIL Termination criterion is satisfied.
    \end{algorithmic}
\end{algorithm}

The estimation of $L_k$ \eqref{eq:estimate-AC-CGM} and the determination of $\gamma_k$ \eqref{eq:stepsize-AC-CGM} are the same as in the AC-PGM.
If $v^k=x^{k-1}$, then $G_k=0$; therefore, the algorithm can be terminated before computing $\tau_k$ in \eqref{eq:tau-AC-CGM}.
Otherwise, $\tau_k$ is well-defined and $\tau_k>0$, and hence $L_k$ is also well-defined by the fact that $x^k\neq x^{k-1}$.

Setting $\beta=\alpha+1/2>1$ and defining $\SS$ in the same way as in the AC-PGM (see \eqref{eq:S}), Assumption \ref{assume:L-smoothness} ensures that exactly the same properties:
\begin{align}
    L_k &\le L,\\
    L_0 &=\gamma_1\le\gamma_2\le\cdots\le\gamma_k\le\max\{L_0,L\}, \label{eq:AC-CGM-gamma-bound}\\
    |\overline{\SS}| &\le\left\lceil\log_\beta\frac{\max\{L_0,L\}}{L_0}\right\rceil\label{eq:failure-bound-AC-CGM}
\end{align}
hold for the AC-CGM as well.

The convergence of Algorithm \ref{alg:AC-CGM} is established as follows.

\begin{theorem}\label{thm:convergence-AC-CGM}
Let $\{x^k\}$ be a sequence generated by Algorithm \ref{alg:AC-CGM} satisfying $v^k\neq x^{k-1}$ for all $k\ge1$.
Suppose that Assumptions \ref{assume:L-smoothness} and \ref{assume:AC-CGM} hold.
Then the following assertions hold.
\begin{enumerate}[(i)]
    \item It holds that
    \begin{equation}\label{eq:sublinear-rate-AC-CGM}
        \min_{1\le l\le k}G_l \le \max\left\{\frac{4\alpha L_0\Delta+2C}{(2\alpha-1)L_0k},\sqrt{\frac{\alpha\max\{L_0,L\}D_g^2(4\alpha L_0\Delta+2C)}{(2\alpha-1)L_0k}}\right\}=\OO(k^{-\frac{1}{2}})
    \end{equation}
    for all $k\ge1$, where $\Delta\coloneqq F(x^0)-F^*$ and $C\coloneqq (\max\{L_0,L\}-L_0)\max_{l \in \overline{\SS}}G_l<\infty$.
    \item The sequence $\{F(x^k)\}$ converges to a certain finite value and any accumulation point of $\{x^k\}$ is an l-stationary point of \eqref{problem:composite}.
\end{enumerate}
\end{theorem}

\begin{proof}
(i) We obtain from \eqref{eq:estimate-AC-CGM} and Assumption \ref{assume:AC-CGM} (i) that
\begin{align}
    F(x^k) &\stackrel{\eqref{eq:estimate-AC-CGM}}{=}f(x^{k-1})+\innerprod{\nabla f(x^{k-1})}{x^k-x^{k-1}}+\frac{L_k}{2}\|x^k-x^{k-1}\|^2+g(x^k)\\
    &\le f(x^{k-1})+\innerprod{\nabla f(x^{k-1})}{x^k-x^{k-1}}+\frac{L_k}{2}\|x^k-x^{k-1}\|^2+(1-\tau_k)g(x^{k-1})+\tau_kg(v^k) \quad (\because \text{the convexity of }g)\\
    &=F(x^{k-1})-\tau_kG_k+\frac{L_k\tau_k^2}{2}\|x^{k-1}-v^k\|^2.
\end{align}
If $\tau_k=1$, which is equivalent to $G_k\ge\alpha\gamma_k\|x^{k-1}-v^k\|^2$, then
\begin{equation}
    F(x^k)\le F(x^{k-1})-G_k+\frac{L_kG_k}{2\alpha\gamma_k}.
\end{equation}
Otherwise, since $\tau_k=G_k/(\alpha\gamma_k\|x^{k-1}-v^k\|^2)$, we have
\begin{equation}
    F(x^k)\le F(x^{k-1})-\tau_kG_k+\frac{L_k\tau_kG_k}{2\alpha\gamma_k}.
\end{equation}
Combining both cases, it holds that
\begin{equation}\label{eq:pseudo-decrease-AC-CGM}
    \left(1-\frac{L_k}{2\alpha\gamma_k}\right)\tau_kG_k+F(x^k)-F(x^{k-1})\le0.
\end{equation}
If $k\in\SS$, by \eqref{eq:pseudo-decrease-AC-CGM} and the definition of $\SS$, we have 
\begin{align}\label{eq:decrease-AC-CGM}
\begin{split}
    F(x^{k-1})-F(x^k) &\stackrel{\eqref{eq:pseudo-decrease-AC-CGM}}{\ge}\left(1-\frac{L_k}{2\alpha\gamma_k}\right)\tau_kG_k\\
    &\ge\frac{2\alpha-1}{4\alpha}\tau_kG_k \qquad (\because \beta\ge L_k/\gamma_k)\\
    &=\frac{2\alpha-1}{4\alpha}\min\left\{G_k,\frac{G_k^2}{\alpha\gamma_k\|x^{k-1}-v^k\|^2}\right\}\\
    &\ge\frac{2\alpha-1}{4\alpha}\min\left\{G_k,\frac{G_k^2}{\alpha\max\{L_0,L\}D_g^2}\right\},
\end{split}
\end{align}
where the last inequality follows from $\gamma_k\le\max\{L_0,L\}$ and $\|x^{k-1}-v^k\|\le D_g$.
On the other hand, $k\notin\SS$ implies $\gamma_{k+1}=\max\{L_0,\ldots,L_k\}=L_k$. 
Thus, it follows from \eqref{eq:pseudo-decrease-AC-CGM} that
\begin{align} \label{eq:AC-CGM-descent-iter}
\begin{split}
    &\frac{2\alpha-1}{2\alpha}\min\left\{G_k,\frac{G_k^2}{\alpha\max\{L_0,L\}D_g^2}\right\}\\
    &\le\frac{2\alpha-1}{2\alpha}\tau_kG_k \qquad (\because \gamma_k\le\max\{L_0,L\}\text{ and }\|x^{k-1}-v^k\|\le D_g)\\
    &=\left(1-\frac{1}{2\alpha}\right)\tau_kG_k\\
    &\le F(x^{k-1})-F(x^k)+\left(\frac{\gamma_{k+1}}{2\alpha\gamma_k}-\frac{1}{2\alpha}\right)\tau_kG_k \qquad (\because \gamma_{k+1}=L_k\text{ and }\eqref{eq:pseudo-decrease-AC-CGM})\\
    &=F(x^{k-1})-F(x^k)+\frac{\gamma_{k+1}-\gamma_k}{2\alpha\gamma_k}\tau_kG_k\\
    &\le F(x^{k-1})-F(x^k)+\frac{\gamma_{k+1}-\gamma_k}{2\alpha L_0}G_k \qquad (\because L_0\le\gamma_k\text{ and }\tau_k\le1).
\end{split}
\end{align}
By summing up, we have
\begin{align}
    &\frac{2\alpha-1}{4\alpha}k\min\left\{\min_{1\le l\le k}G_l,\frac{\min_{1\le l\le k}G_l^2}{\alpha\max\{L_0,L\}D_g^2}\right\}\\
    &=\frac{2\alpha-1}{4\alpha}k\min_{1\le l\le k}\min\left\{G_l,\frac{G_l^2}{\alpha\max\{L_0,L\}D_g^2}\right\}\\
    &\le\frac{2\alpha-1}{4\alpha}\sum_{l\in[k]}\min\left\{G_l,\frac{G_l^2}{\alpha\max\{L_0,L\}D_g^2}\right\}\\
    &\le F(x^0)-F(x^k)+\sum_{l\in[k]\cap\overline{\SS}}\frac{\gamma_{l+1}-\gamma_l}{2\alpha L_0}G_l \qquad (\because \eqref{eq:decrease-AC-CGM}\text{ and }\eqref{eq:AC-CGM-descent-iter})\\
    &\le \Delta + \frac{\max_{l \in \overline{\SS}}G_l}{2\alpha L_0}\sum_{l=1}^k(\gamma_{l+1}-\gamma_l)\qquad (\because F(x^k)\ge F^* \text{ and the monotonicity of }\{\gamma_l\})\\
    & = \Delta + \frac{\max_{l \in \overline{\SS}}G_l}{2\alpha L_0}(\gamma_{k+1}-\gamma_1)\\
    & \stackrel{\eqref{eq:AC-CGM-gamma-bound}}{\leq} \Delta + \frac{(\max\{L_0,L\}-L_0)\max_{l \in \overline{\SS}}G_l}{2\alpha L_0}
    =\Delta+\frac{C}{2\alpha L_0}.
\end{align}
Rearranging this yields
\begin{equation}
    \min_{1\le l\le k}G_l\le\frac{4\alpha L_0\Delta+2C}{(2\alpha-1)L_0k}
\end{equation}
or
\begin{equation}
    \min_{1\le l\le k}G_l^2\le\frac{\alpha\max\{L_0,L\}D_g^2\{4\alpha L_0\Delta+2C\}}{(2\alpha-1)L_0k}.
\end{equation}
Combining them proves the assertion (i).

(ii) By \eqref{eq:decrease-AC-CGM} and the finiteness of $\overline{\SS}$, it holds that
\begin{equation}\label{eq:sufficient-decrease-AC-CGM}
    \frac{2\alpha-1}{2\alpha}\min\left\{G_k,\frac{G_k^2}{\alpha\max\{L_0,L\}D_g^2}\right\}\le F(x^{k-1})-F(x^k)
\end{equation}
for all sufficiently large $k$.
Thus, we see from the boundedness from below of $F$ that $\{F(x^k)\}$ converges, and hence
\begin{equation}\label{eq:diff-vanish-AC-CGM}
    G_k=G(x^{k-1})\to0.
\end{equation}
Let $\{x^k\}_K$ be a subsequence of $\{x^k\}$ converging to some point $x^*$.
By the lower semicontinuity of $G$, we have
\begin{equation}
    G(x^*)\le\liminf_{k\to\infty, k\in K}G(x^k)=0,
\end{equation}
which is the desired result.
\end{proof}

By adding a few additional assumptions, we obtain the following complexity bound.

\begin{corollary}\label{cor:complexity-AC-CGM}
In addition to the same assumptions as in Theorem \ref{thm:convergence-AC-CGM}, we suppose that $\dom g$ is a closed set and $g$ is continuous on $\dom g$.
Then, Algorithm \ref{alg:AC-CGM} finds an $\varepsilon$-stationary point satisfying $G_k\le\varepsilon$ within
\begin{equation}
     \max\left\{\frac{4\alpha L_0\Delta+2(\max\{L_0,L\}-L_0)G_g}{(2\alpha-1)L_0\varepsilon},\frac{\alpha\max\{L_0,L\}D_g^2\{4\alpha L_0\Delta+2(\max\{L_0,L\}-L_0)G_g\}}{(2\alpha-1)L_0\varepsilon^2}\right\}
\end{equation}
iterations, where $G_g\coloneqq\sup_{x\in\dom g}G(x)$.
\end{corollary}

\begin{proof}
Since $\dom g$ is compact and $\nabla f$ and $g$ are continuous on $\dom g$, it holds that
\begin{equation}
    G_g=\sup_{x\in\dom g}G(x)=\sup_{x,v\in\dom g}\left\{\innerprod{\nabla f(x)}{x-v}+g(x)-g(v)\right\}<\infty.
\end{equation}
From Theorem \ref{thm:convergence-AC-CGM} with $C\leq (\max\{L_0,L\}-L_0)G_g$, we have the desired result.
\end{proof}

Note that if $g$ is the indicator function of a compact convex set, then the assumptions of Corollary \ref{cor:complexity-AC-CGM} are satisfied.
The complexity bound in Corollary \ref{cor:complexity-AC-CGM} is dominated by
$$
\OO\left(
\frac{L\Delta D_g^2}{\varepsilon^2} + \frac{L^2G_g D_g^2}{L_0 \varepsilon^2}
\right)
$$
when $L_0<L$.
Moreover, in the case when $L_0\geq L$, the AC-CGM results in the conditional gradient method with well-known stepsize selection $\tau_k=\min\left\{1,\frac{G_k}{\alpha L_0 \|x^{k-1}-v^k\|^2}\right\}$ yielding the complexity guarantee of $\OO(L D_g^2 \Delta/\varepsilon^2)$, which is compatible with known results \cite{lacostejulien2016,braun2023conditional}.

\subsection{Riemannian gradient method}\label{subsec:RGM}
Lastly, we consider solving the following optimization problem
\begin{equation}\label{problem:Riemannian}
    \underset{x\in\MM}{\mbox{minimize}} \quad f(x),
\end{equation}
where $\MM$ is a smooth Riemannian manifold equipped with a Riemann metric $\innerprod{\cdot}{\cdot}_x$ and $f:\MM\to\RR$ is of class $C^1$ and is bounded from below.

We now prepare the notions related to Riemannian manifolds to be used below (see \citep{absil2008optimization,sato2021riemannian,boumal2023introduction} for details).
The tangent space of the manifold $\MM$ at $x$ and the tangent bundle of $\MM$ are denoted by $T_x\MM$ and $T\MM$, respectively.
Let $R:T\MM\to\MM$ be a retraction on $\MM$, that is, for all $x\in\MM$, it holds that (i) $R_x(0_x)=x$ where $0_x$ denotes the zero element of $T_x\MM$, and $\mathrm{D}R_x(0_x)$ is the identity map on $T_x\MM$ where $\mathrm{D}R_x(0_x)$ is the differential of $R_x$ at $0_x$.
The gradient field of $f$ at $x\in\MM$ is denoted by $\grad f(x)\in T_x\MM$.
It is not hard to see that the function $x\mapsto\|\grad f(x)\|_x$ is continuous because $f$ is of class $C^1$, where $\|\xi\|_x\coloneqq\sqrt{\innerprod{\xi}{\xi}_x}$ for $\xi\in T\MM_x$.

We make the following assumptions for the optimization problem \eqref{problem:Riemannian}.

\begin{assumption}\label{assume:L-retraction-smoothness}
There exists $L>0$ such that
\begin{equation}\label{eq:Riemannian-descent-lemma}
    f(R_x(\xi))\le f(x)+\innerprod{\grad f(x)}{\xi}_x+\frac{L}{2}\|\xi\|_x^2
\end{equation}
holds for any $(x,\xi)\in T\MM$.
\end{assumption}

Assumption \ref{assume:L-retraction-smoothness} is called $L$-retraction-smoothness and is often used in the analysis of first-order methods on Riemannian manifolds \citep{boumal2019global,huang2022riemannian,huang2023inexact}.
If $\MM$ is a compact Riemannian submanifold of a Euclidean space $\EE$, then the $L$-smoothness of $f:\EE\to\RR$ on $\convhull(\MM)$ implies the $L$-retraction smoothness of $f_{\mid\MM}$ \citep[Lemma 2.7]{boumal2019global}.
Commonly used manifolds such as the sphere $S^{n-1}\coloneqq\{x\in\RR^n\mid x^\top x=1\}$, and more generally the Stiefel manifold $\mathrm{St}(n,r)\coloneqq\{X\in\RR^{n\times r}\mid X^\top X=I\}$, are compact Riemannian submanifolds of $\RR^n$ and $\RR^{n\times r}$, respectively.

It is known that any local minimizer $x^*$ of \eqref{problem:Riemannian} satisfies $\grad f(x^*)=0_{x^*}$.
We call $x^*\in\MM$ satisfying $\grad f(x^*)=0_{x^*}$ a stationary point of \eqref{problem:Riemannian} and use $\|\grad f(x)\|_{x}$ as an optimality measure.

The auto-conditioned Riemannian gradient method (AC-RGM) is summarized in Algorithm \ref{alg:AC-RGM}.

\begin{algorithm}[H]
\caption{Auto-conditioned Riemannian gradient method (AC-RGM)}
    \label{alg:AC-RGM}
    \begin{algorithmic}
    \STATE {\bfseries Input:} $x^0\in\dom g,~ \alpha>\frac{1}{2},~ L_0>0$, and $k=1$.
    \REPEAT
    \STATE Compute
    \begin{align}
        \gamma_k &=\max\{L_0,\ldots,L_{k-1}\},\label{eq:stepsize-AC-RGM}\\
        \tau_k &=\frac{1}{\alpha\gamma_k},\label{eq:tau-AC-RGM}\\
        x^k &=R_{x^{k-1}}(-\tau_k\grad f(x^{k-1})),\label{eq:iterate-AC-RGM}\\
        L_k &=\frac{2\left(f(x^k)-f(x^{k-1})-\innerprod{\grad f(x^{k-1})}{-\tau_k\grad f(x^{k-1})}_{x^{k-1}}\right)}{\|\tau_k\grad f(x^{k-1})\|_{x^{k-1}}^2}.\label{eq:estimate-AC-RGM}
    \end{align}
    \STATE Set $k\leftarrow k+1$.
    \UNTIL Termination criterion is satisfied.
    \end{algorithmic}
\end{algorithm}

The estimation of $L_k$ \eqref{eq:estimate-AC-RGM} is similar to that of the AC-PGM and AC-CGM.
On the other hand, the determination of $\gamma_k$ \eqref{eq:stepsize-AC-RGM} is exactly the same.
If $\grad f(x^{k-1})=0_{x^{k-1}}$, equivalently, $x^{k-1}$ is a stationary point; therefore, the algorithm can be terminated before computing $L_k$ in \eqref{eq:estimate-AC-RGM}.
Otherwise, $L_k$ is well-defined because of $\grad f(x^{k-1})\neq0_{x^{k-1}}$.

Although the determination of $L_k$ is slightly different from that in the AC-PGM and AC-CGM, by setting $\beta=\alpha+1/2>1$ and defining $\SS$ in the same way as in those cases (see \eqref{eq:S}), the following properties likewise hold for Algorithm \ref{alg:AC-RGM}:
\begin{align}
    L_k &\le L,\\
    L_0 &=\gamma_1\le\gamma_2\le\cdots\le\gamma_k\le\max\{L_0,L\} \label{eq:AC-RGM-gamma-bound},\\
    |\overline{\SS}| &\le\left\lceil\log_\beta\frac{\max\{L_0,L\}}{L_0}\right\rceil.\label{eq:failure-bound-AC-RGM}
\end{align}

Convergence result of the AC-RGM is obtained as follows.

\begin{theorem}\label{thm:convergence-AC-RGM}
Let $\{x^k\}$ be a sequence generated by Algorithm \ref{alg:AC-RGM} satisfying $\grad f(x^{k-1})\neq0_{x^{k-1}}$ for all $k\ge1$.
Suppose that Assumption \ref{assume:L-retraction-smoothness} holds.
Then the following assertions hold.
\begin{enumerate}[(i)]
    \item It holds that
    \begin{equation}\label{eq:sublinear-rate-AC-RGM}
        \min_{1\le l\le k}\|\grad f(x^{l-1})\|_{x^{l-1}} \le \sqrt{\frac{2\max\{L_0,L\}(2\alpha^2L_0^2\Delta+C)}{(2\alpha-1)L_0^2k}}=\OO(k^{-\frac{1}{2}})
    \end{equation}
    for all $k\ge1$, where $f^*\coloneqq\inf_{x\in\MM}f(x)$, $\Delta\coloneqq f(x^0)-f^*$, and
    $$C=(\max\{L_0,L\}-L_0)\max_{l \in \overline{\SS}}\|\grad f(x^{l-1})\|_{x^{l-1}}^2<\infty.$$
    \item The sequence $\{f(x^k)\}$ converges to a certain finite value and any accumulation point of $\{x^k\}$ is a stationary point of \eqref{problem:Riemannian}.
\end{enumerate}
\end{theorem}

\begin{proof}
(i) We obtain from \eqref{eq:estimate-AC-RGM} that
\begin{align}\label{eq:pseudo-decrease-AC-RGM}
\begin{split}
    f(x^k) &=f(x^{k-1})+\innerprod{\grad f(x^{k-1})}{-\tau_k\grad f(x^{k-1})}_{x^{k-1}}+\frac{L_k}{2}\|\tau_k\grad f(x^{k-1})\|_{x^{k-1}}^2\\
    &=f(x^{k-1})-\tau_k\left(1-\frac{L_k}{2\alpha\gamma_k}\right)\|\grad f(x^{k-1})\|_{x^{k-1}}^2\\
    &=f(x^{k-1})-\frac{1}{\alpha\gamma_k}\left(1-\frac{L_k}{2\alpha\gamma_k}\right)\|\grad f(x^{k-1})\|_{x^{k-1}}^2
\end{split}
\end{align}
If $k\in\SS$, by \eqref{eq:pseudo-decrease-AC-RGM} and the definition of $\SS$, we have 
\begin{align}\label{eq:decrease-AC-RGM}
\begin{split}
    f(x^{k-1})-f(x^k) &\stackrel{\eqref{eq:pseudo-decrease-AC-RGM}}{=}\frac{1}{\alpha\gamma_k}\left(1-\frac{L_k}{2\alpha\gamma_k}\right)\|\grad f(x^{k-1})\|_{x^{k-1}}^2\\
    &\ge\frac{1}{\alpha\gamma_k}\frac{2\alpha-1}{4\alpha}\|\grad f(x^{k-1})\|_{x^{k-1}}^2 \qquad (\because \beta\ge L_k/\gamma_k)\\
    &\ge\frac{2\alpha-1}{4\alpha^2\max\{L_0,L\}}\|\grad f(x^{k-1})\|_{x^{k-1}}^2 \qquad (\because \gamma_k\le\max\{L_0,L\}).
\end{split}
\end{align}
On the other hand, $k\notin\SS$ implies $\gamma_{k+1}=\max\{L_0,\ldots,L_k\}=L_k$. 
Thus, it follows from \eqref{eq:pseudo-decrease-AC-RGM} that
\begin{align}
    &\frac{2\alpha-1}{2\alpha^2\max\{L_0,L\}}\|\grad f(x^{k-1})\|_{x^{k-1}}^2\\
    &=\frac{1}{\alpha\max\{L_0,L\}}\left(1-\frac{1}{2\alpha}\right)\|\grad f(x^{k-1})\|_{x^{k-1}}^2\\
    &\le\frac{1}{\alpha\gamma_k}\left(1-\frac{1}{2\alpha}\right)\|\grad f(x^{k-1})\|_{x^{k-1}}^2 \qquad (\because \gamma_k\le\max\{L_0,L\})\\
    &=f(x^{k-1})-f(x^k)+\frac{1}{\alpha\gamma_k}\left(\frac{\gamma_{k+1}}{2\alpha\gamma_k}-\frac{1}{2\alpha}\right)\|\grad f(x^{k-1})\|_{x^{k-1}}^2 \qquad (\because \gamma_{k+1}=L_k\text{ and }\eqref{eq:pseudo-decrease-AC-RGM})\\
    &\le f(x^{k-1})-f(x^k)+\frac{\gamma_{k+1}-\gamma_k}{2\alpha^2L_0^2}\|\grad f(x^{k-1})\|_{x^{k-1}}^2 \qquad (\because L_0\le\gamma_k).
\end{align}
By summing up, we have
\begin{align}
    &\frac{2\alpha-1}{4\alpha^2\max\{L_0,L\}}k\min_{1\le l\le k}\|\grad f(x^{l-1})\|_{x^{l-1}}^2\\
    &\le \frac{2\alpha-1}{4\alpha^2\max\{L_0,L\}}\sum_{i\in[k]}\|\grad f(x^{l-1})\|_{x^{l-1}}^2\\
    &\le f(x^0)-f(x^k)+\sum_{l\in[k]\cap\overline{\SS}}\frac{\gamma_{l+1}-\gamma_l}{2\alpha^2L_0^2}\|\grad f(x^{l-1})\|_{x^{l-1}}^2\\
    &\le \Delta+\sum_{l\in[k]\cap\overline{\SS}}\frac{\gamma_{l+1}-\gamma_l}{2\alpha^2L_0^2}\|\grad f(x^{l-1})\|_{x^{l-1}}^2 \qquad (\because f(x^k)\ge f^*).
\end{align}
Rearranging this yields
\begin{equation}
    \min_{1\le l\le k}\|\grad f(x^{l-1})\|_{x^{l-1}} \le \sqrt{\frac{\max\{L_0,L\}\{4\alpha^2L_0^2\Delta+2\sum_{l\in[k]\cap\overline{\SS}}(\gamma_{l+1}-\gamma_l)\|\grad f(x^{l-1})\|_{x^{l-1}}^2\}}{(2\alpha-1)L_0^2k}}
\end{equation}
The assertion (i) follows by this inequality combined with the following bound.
\begin{align}
\sum_{l\in[k]\cap\overline{\SS}}(\gamma_{l+1}-\gamma_l)\|\grad f(x^{l-1})\|_{x^{l-1}}^2
&\leq
\max_{l \in \overline{\SS}}\|\grad f(x^{l-1})\|_{x^{l-1}}^2\sum_{l=1}^k(\gamma_{l+1}-\gamma_l) \quad (\because \text{ the monotonicity of } \{\gamma_l\})\\
&=
\max_{l \in \overline{\SS}}\|\grad f(x^{l-1})\|_{x^{l-1}}^2(\gamma_{k+1}-\gamma_1)\leq C\quad (\because \gamma_{k+1}\leq \max\{L_0,L\},\gamma_1=L_0).
\end{align} 

(ii) By the finiteness of $\overline{\SS}$, \eqref{eq:decrease-AC-RGM} holds for all sufficiently large $k$.
Thus, we see from the boundedness from below of $f$ that $\{f(x^k)\}$ converges, and hence
\begin{equation}\label{eq:diff-vanish-AC-RGM}
    \|\grad f(x^{k-1})\|_{x^{k-1}}^2\to0.
\end{equation}
Let $\{x^k\}_K$ be a subsequence of $\{x^k\}$ converging to some point $x^*$.
The continuity of $x\mapsto\|\grad f(x)\|_x$ yields
\begin{equation}
   \|\grad f(x^*)\|_{x^*}=\lim_{k\to\infty, k\in K}\|\grad f(x^k)\|_{x^k}=0,
\end{equation}
which implies that $x^*$ is a stationary point of \eqref{problem:Riemannian}.
\end{proof}

In the presence of the boundedness of the gradient of $f$, we have the following complexity bound as an immediate consequence of Theorem \ref{thm:convergence-AC-RGM}.

\begin{corollary}\label{cor:complexity-AC-RGM}
In addition to the same assumptions as in Theorem \ref{thm:convergence-AC-RGM}, we suppose that
\begin{equation}
    G_f\coloneqq\sup_{x\in\MM}\|\grad f(x)\|_x<\infty.
\end{equation}
Then, Algorithm \ref{alg:AC-RGM} finds an $\varepsilon$-stationary point satisfying $\|\grad f(x^{k-1})\|_{x^{k-1}}\le\varepsilon$ within
\begin{equation}
     \frac{2\max\{L_0,L\}\{2\alpha^2L_0^2\Delta+(\max\{L_0,L\}-L_0)G_f^2\}}{(2\alpha-1)L_0^2\varepsilon^2}
\end{equation}
iterations.
\end{corollary}

If $\MM$ is compact, then by continuity, the function $x\mapsto\|\grad f(x)\|_x$ is automatically bounded.
Corollary \ref{cor:complexity-AC-RGM} provides an iteration complexity bound  that matches the order of $\varepsilon$ obtained by \citet{boumal2019global} for the Riemannian gradient methods with constant stepsize and with backtracking Armijo linesearch, under Assumption \ref{assume:L-retraction-smoothness}.
Considering the case $\MM=\EE$, Corollary \ref{cor:complexity-AC-RGM} also provides the complexity bound of a linesearch-free steepest descent method for unconstrained smooth optimization problems with bounded gradients.

\section{Numerical examples}\label{sec:numerical}
To demonstrate empirical performance of the auto-conditioned stepsize strategy, we conduct two numerical experiments.
In the first, we make a comparison with a constant stepsize strategy, and in the second, with linesearch strategies.
All the algorithms were implemented in MATLAB R2023b, and all the computations were conducted on a Windows computer with Intel Core i7-1355U 2.60GHz processor and 16GB RAM.

\subsection{Comparison with constant stepsize}
We first compare the AC-PGM with a proximal gradient method employing constant stepsize.
The following regularized logistic regression problem is considered:
\begin{equation}
    \underset{x\in\RR^n}{\mbox{minimize}} \quad \underbrace{\frac{1}{m}\sum_{i\in[m]}-\log(1+e^{-b_i(a_i^\top x)})+\frac{\lambda_1}{2}\|x\|_2^2}_{f(x)}+\underbrace{\lambda_2T_\kappa(x)}_{g(x)},
\end{equation}
where $b_i\in\{-1,1\}$ and $a_i\in\RR^n$ for $i\in[m]$ are the given data, $\lambda_1,\lambda_2>0$ are the regularization parameter, and $\|x\|_2\coloneqq\sqrt{\sum_{i\in[n]}x_i^2}$.
The function $T_\kappa$, referred to as the trimmed $\ell_1$ norm, is defined by
\begin{equation}
    T_\kappa(x)\coloneqq |x_{\langle1\rangle}|+\cdots+|x_{\langle n-\kappa\rangle}|=\min_{\substack{\Lambda\subset[n]\\|\Lambda|=n-\kappa}}\sum_{i\in\Lambda}|x_i|,
\end{equation}
where $|x_{\langle 1\rangle}|\le|x_{\langle 2\rangle}|\le\cdots\le|x_{\langle n\rangle}|$ and $\kappa\in\{1,\ldots,n-1\}$.
The trimmed $\ell_1$ norm is a nonconvex nonsmooth function introduced by \citet{luo2013new} and \citet{huang2015two} to obtain a more clear-cut sparse solution than the $\ell_1$ norm.
The trimmed $\ell_1$ norm is known as an exact penalty function of the cardinality constraint $\|x\|_0\le\kappa$, where $\|x\|_0$ is the number of the nonzero elements of $x$ (see \citep{yagishita2022exact} and references therein).

\begin{figure}[htbp]
  \centering
  \begin{minipage}{0.45\columnwidth}
    \centering
    \includegraphics[width=\columnwidth]{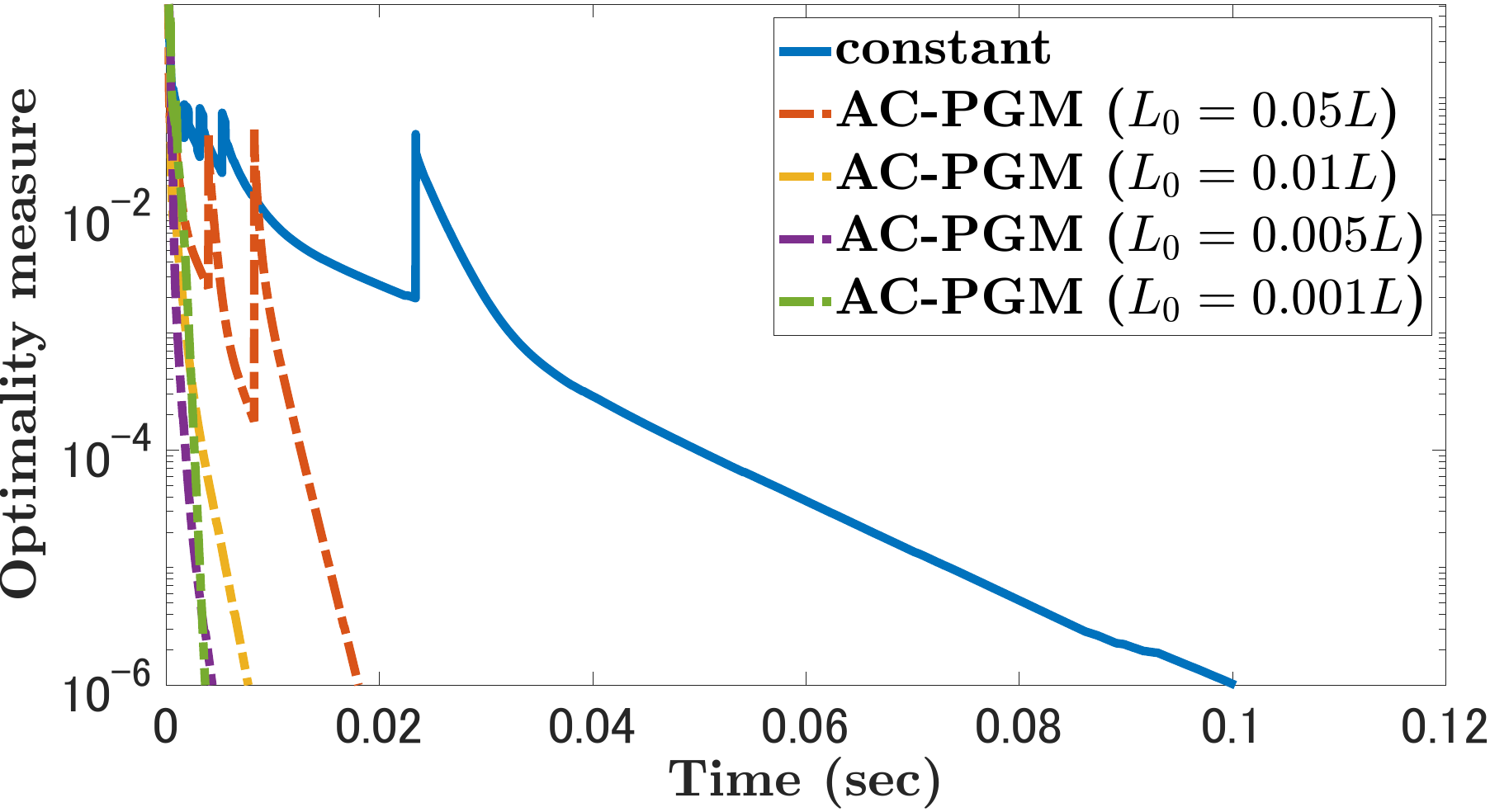}
  \end{minipage}
  \hspace{1em}
  \begin{minipage}{0.45\columnwidth}
    \centering
    \includegraphics[width=\columnwidth]{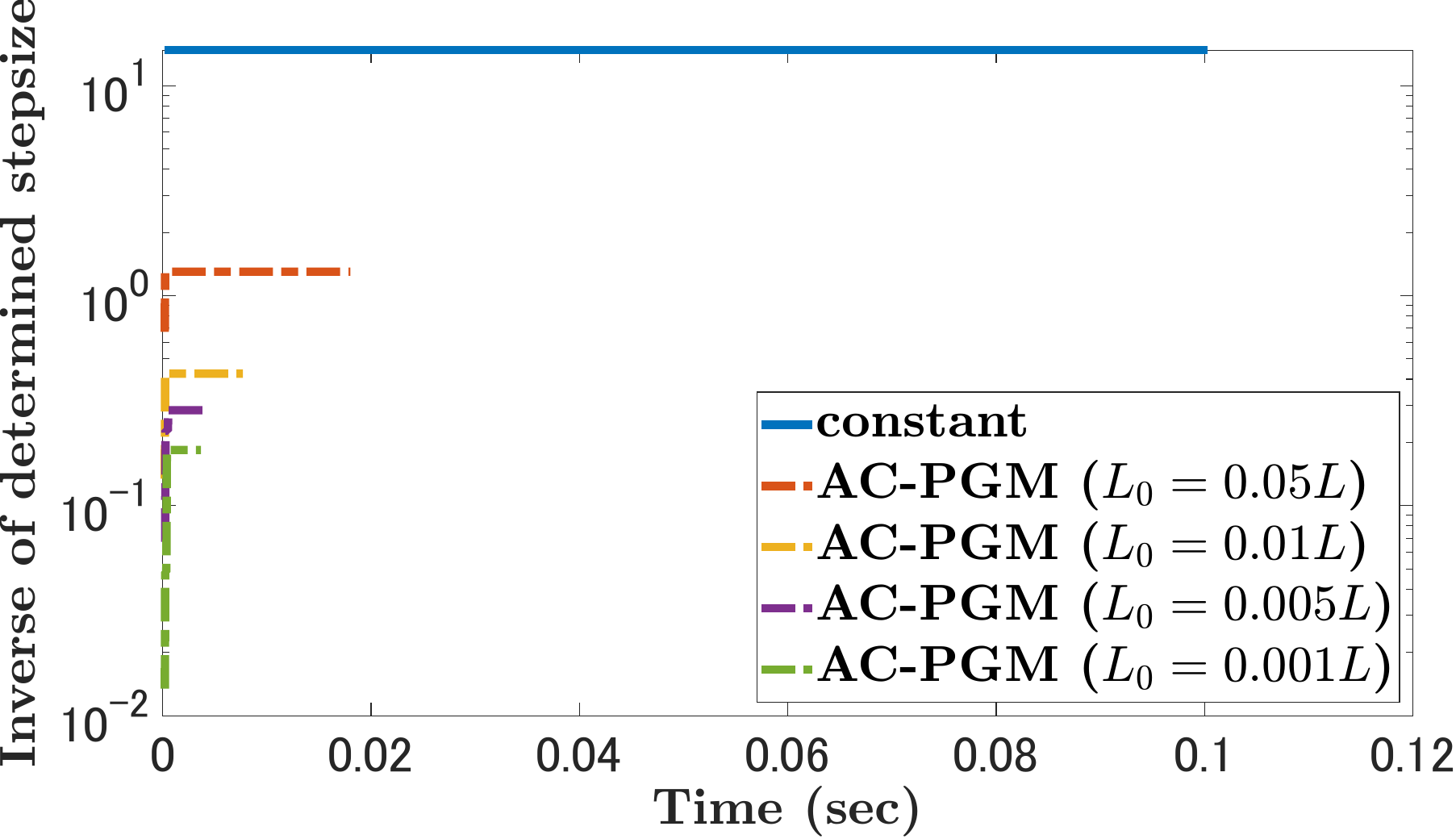}
  \end{minipage}
  
  \vspace{1em}
  (a) Sonar ($m=208, n=60$)
  \vspace{1em}
  
  \begin{minipage}{0.45\columnwidth}
    \centering
    \includegraphics[width=\columnwidth]{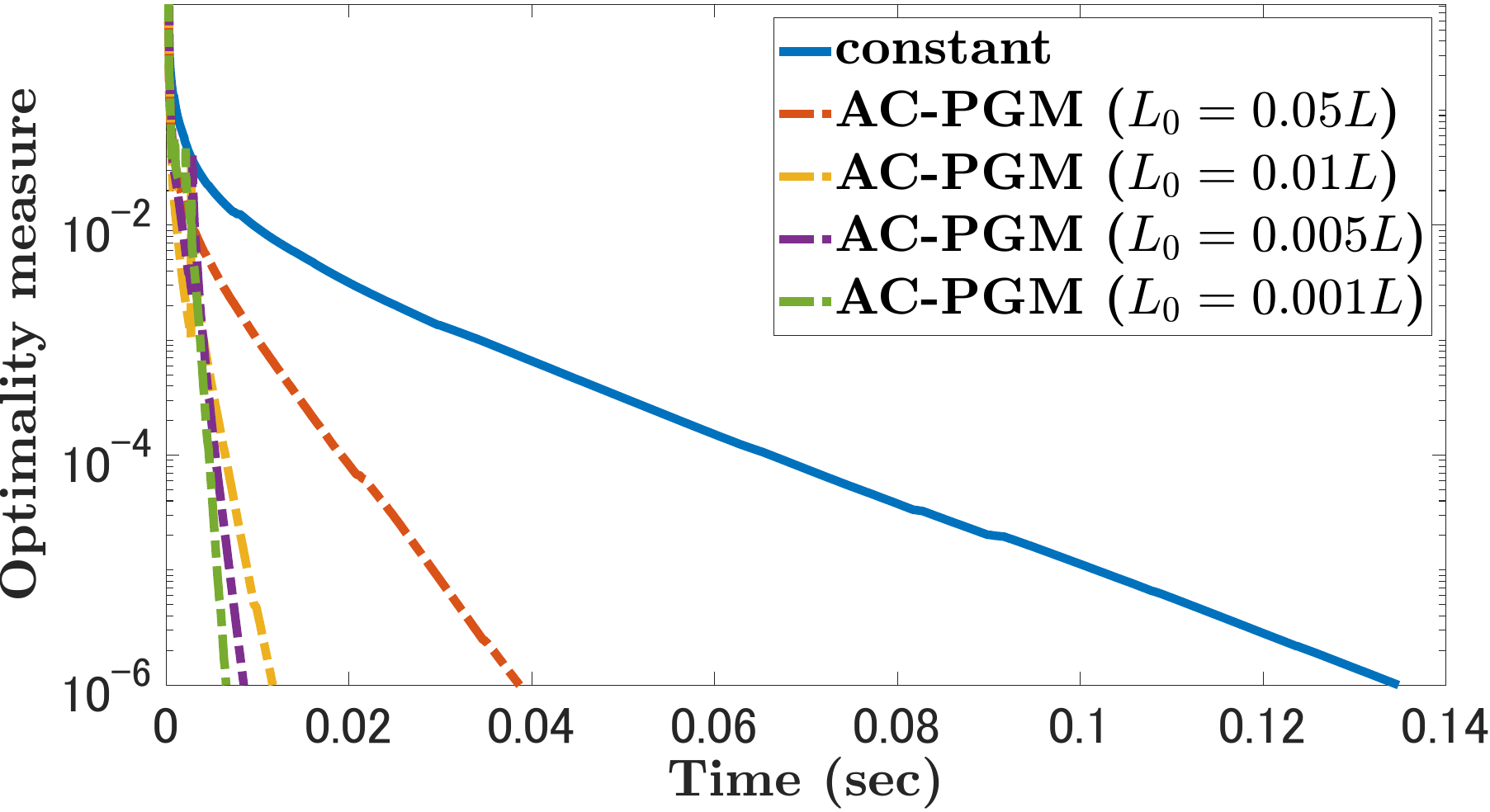}
  \end{minipage}
  \hspace{1em}
  \begin{minipage}{0.45\columnwidth}
    \centering
    \includegraphics[width=\columnwidth]{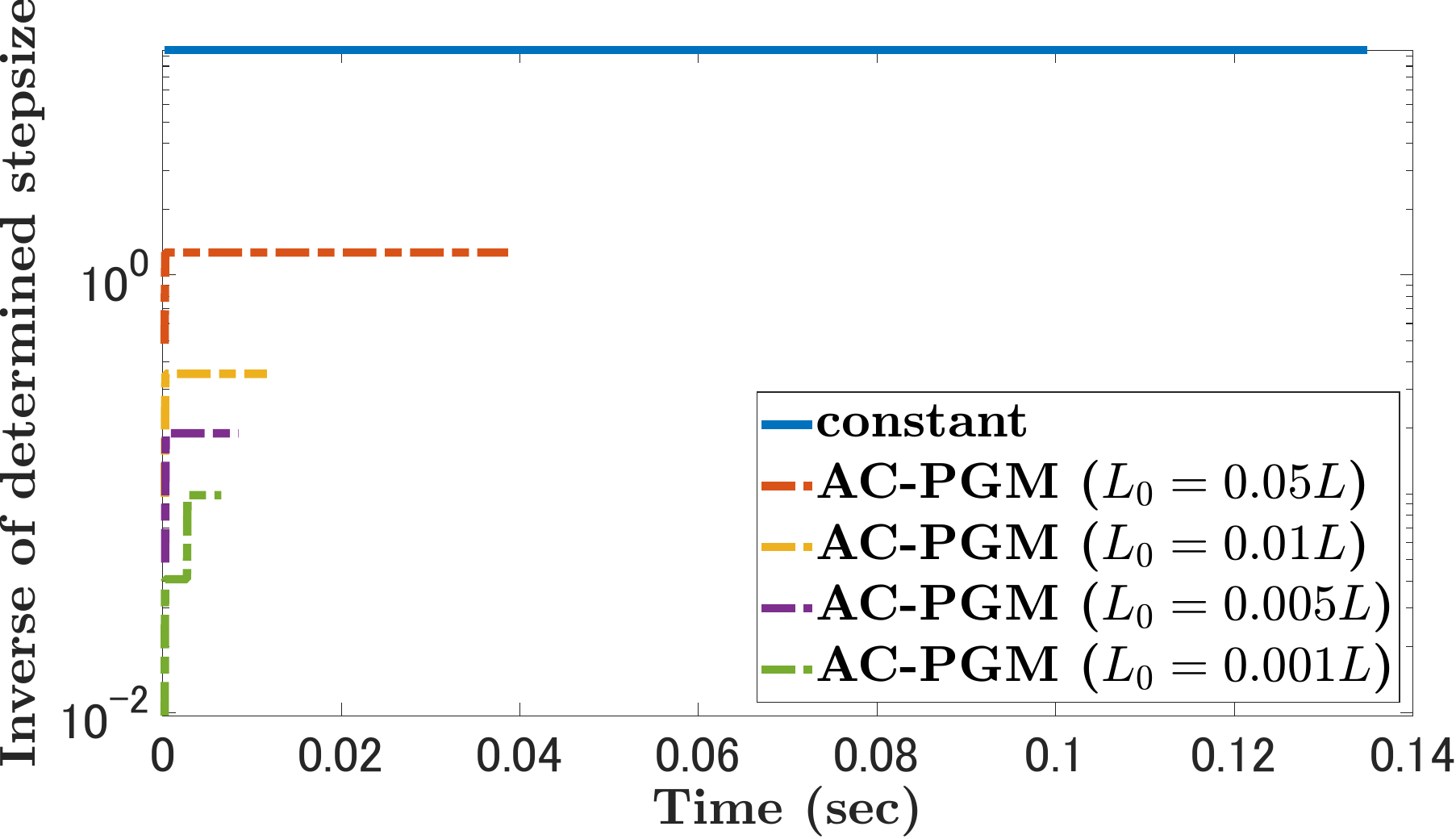}
  \end{minipage}

  \vspace{1em}
  (b) Ionosphere ($m=351, n=33$)
  \vspace{1em}
  
  \begin{minipage}{0.45\columnwidth}
    \centering
    \includegraphics[width=\columnwidth]{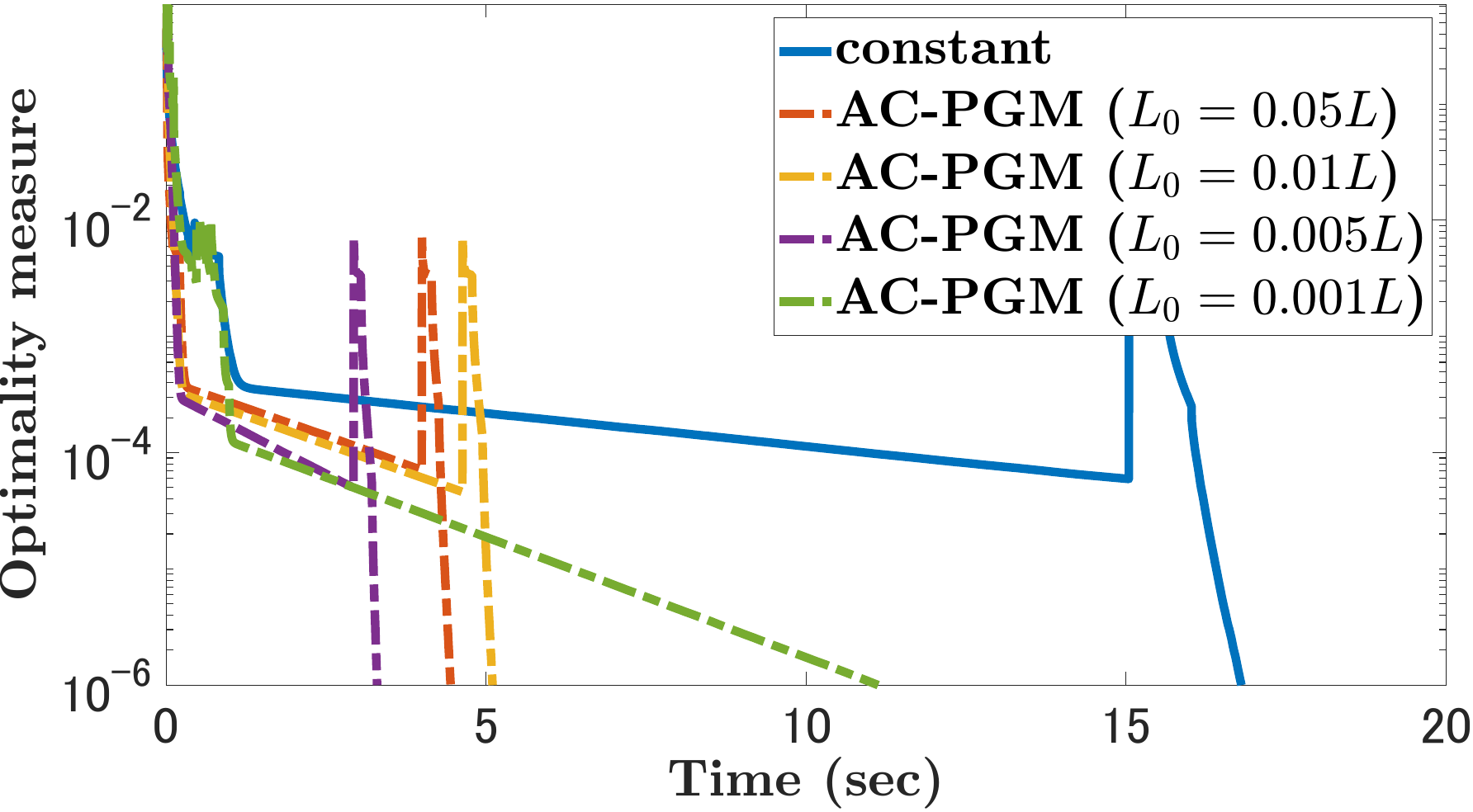}
  \end{minipage}
  \hspace{1em}
  \begin{minipage}{0.45\columnwidth}
    \centering
    \includegraphics[width=\columnwidth]{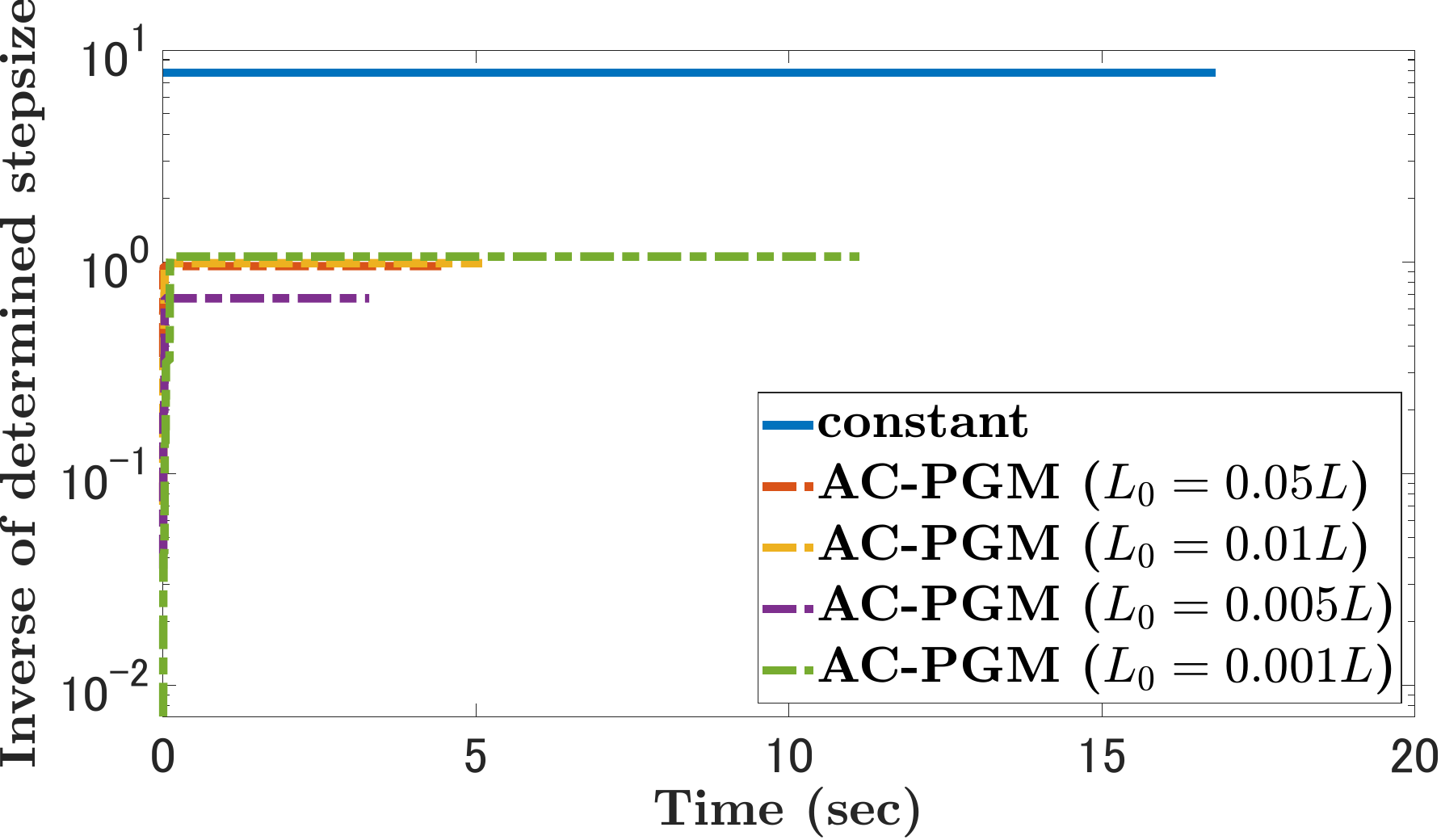}
  \end{minipage}

  \vspace{1em}
  (c) Madelon ($m=2000, n=500$)
  \vspace{1em}
  
  \begin{minipage}{0.45\columnwidth}
    \centering
    \includegraphics[width=\columnwidth]{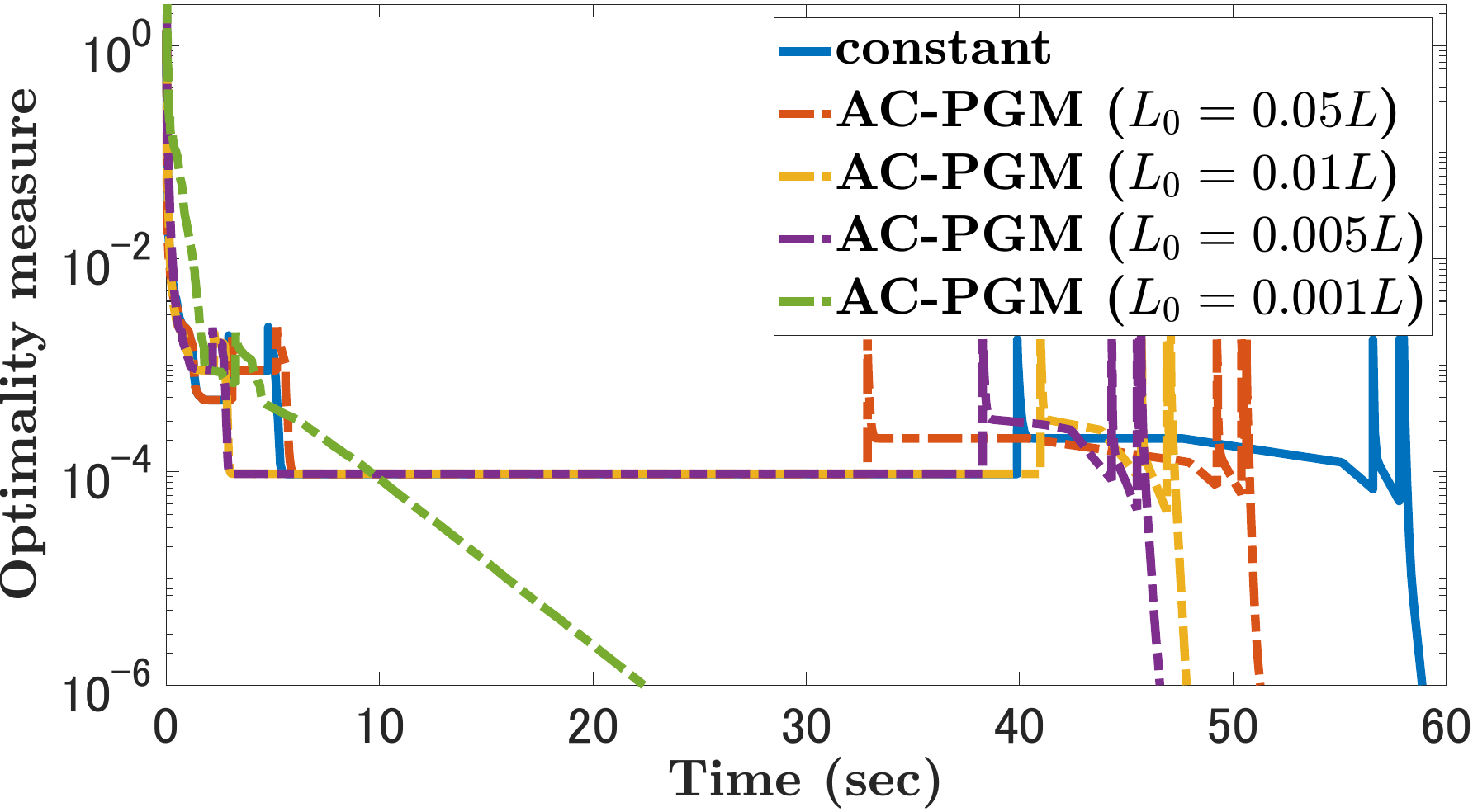}
  \end{minipage}
  \hspace{1em}
  \begin{minipage}{0.45\columnwidth}
    \centering
    \includegraphics[width=\columnwidth]{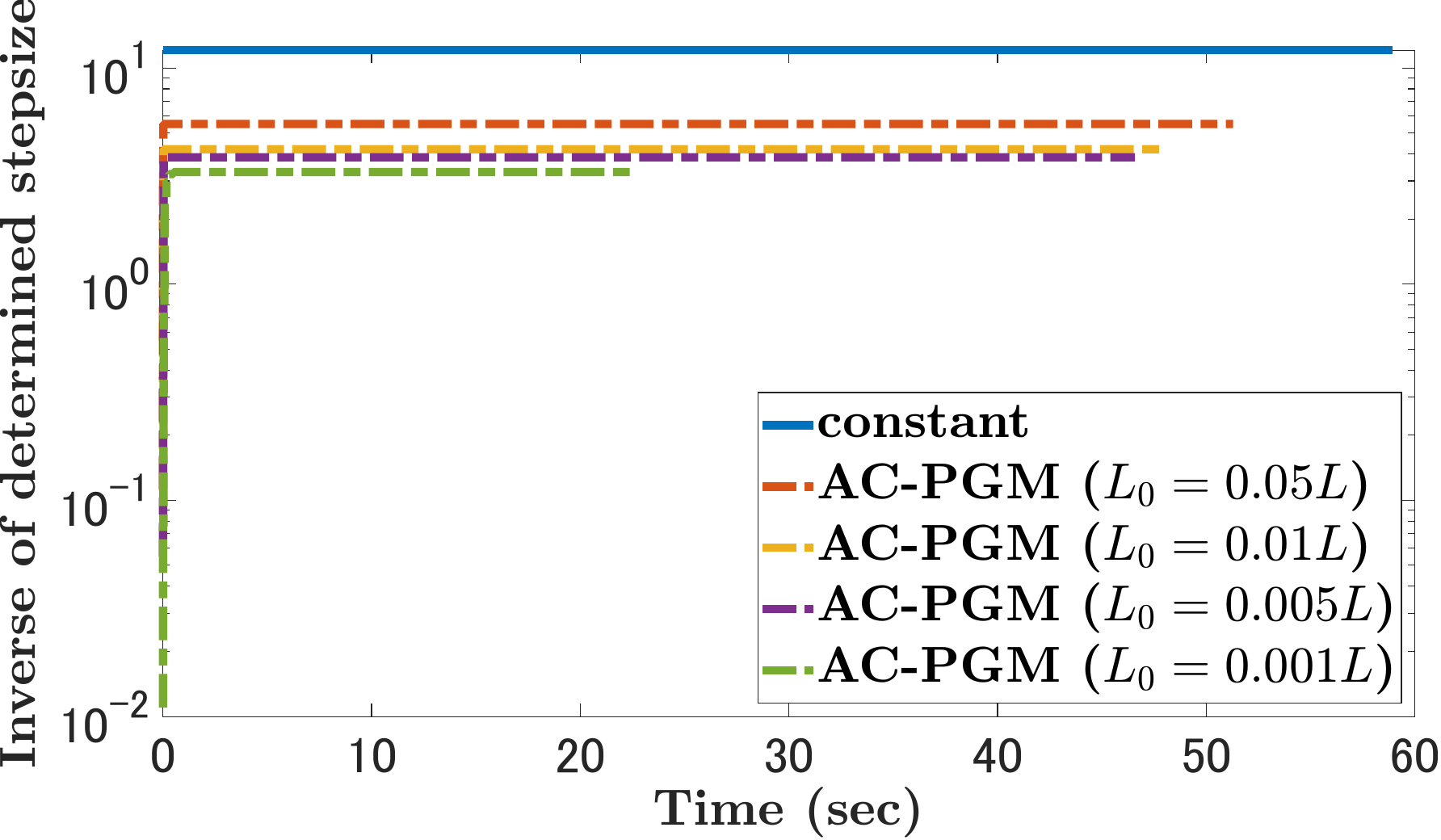}
  \end{minipage}

  \vspace{1em}
  (d) Mushrooms ($m=8124, n=111$)
  
  \label{fig:vs-constant}
  \caption{The convergence behaviors and the inverse of the determined stepsizes of each algorithm on the four datasets.}
\end{figure}

We use four datasets obtained from the LIBSVM\footnote{See \url{https://www.csie.ntu.edu.tw/~cjlin/libsvmtools/datasets/}.}.
In the problem setting, the parameters are chosen as $\lambda_1=10^{-2}/m,~ \lambda_2=10/m$, and $\kappa=10$.
As the upper curvature parameter of $f$ can be estimated in closed form as $L=\frac{\|A\|_{\mathrm{op}}^2}{4m}+\lambda_2$, where $\|A\|_{\mathrm{op}}$ is the operator norm of $A$, we employ $\gamma=1.1L$ as (the inverse of) the constant stepsize.
For the AC-PGM, we set $L_0=\theta L$ with $\theta\in\{0.05,0.01,0.005,0.001\}$ and $\alpha=1.1$.
For both algorithms, the initial point is set as the origin.

The convergence behavior of each algorithm on the four datasets is shown on the left side of Figure \ref{fig:vs-constant}, while the inverse of the stepsizes determined at each iteration are plotted on the right side.
Since the AC-PGM is less conservative than the proximal gradient method with the constant stepsize, it can adopt larger stepsizes and consequently achieves faster convergence.
The results on the Madelon dataset indicate that adopting a smaller $L_0$, i.e., a larger initial stepsize, is not necessarily effective. 
This is likely because choosing an excessively large stepsize leads to estimating the Lipschitz constant between two distant points, which in turn causes the stepsize to shrink in subsequent iterations.
The results on the Mushrooms dataset also show that when the stepsizes of the AC-PGM is close to the constant stepsize, the performance gap with the proximal gradient method with the constant stepsize becomes small.
From the figure \ref{fig:vs-constant}, the optimality measure of the AC-PGM appears to converge linearly; indeed, the KL exponent of the objective function of the problem is $1/2$ \citep[Corollary 5.1]{li2017calculus}, and hence Remark \ref{rem:KL} confirms that this is theoretically valid.

\subsection{Comparison with Armijo linesearch}
In the second experiment, we compare the AC-RGM with Riemannian gradient methods with Armijo-type linesearch.
The following optimization problem on the Stiefel manifold is considered:
\begin{equation}
    \underset{X\in\mathrm{St}(n,r)}{\mbox{minimize}} \quad \tr(X^\top AXN),
\end{equation}
where $A\in\RR^{n\times n}$ is a symmetric matrix and $N\in\RR^{r\times r}$ is a diagonal matrix with diagonal elements $r,r-1,\ldots,2,1$.
Here, we consider the standard inner product as the Riemann metric, namely, $\innerprod{Z_1}{Z_2}_X=\tr(Z_1^\top Z_2)$ for $X\in\mathrm{St}(n,r)$ and $Z_1,Z_2\in T_X\mathrm{St}(n,r)$.
As a retraction on the Stiefel manifold, we use the one based on the QR decomposition.
Specifically, for $X\in\mathrm{St}(n,r)$ and $Y\in T_X\mathrm{St}(n,r)$, the retraction returns the Q-factor of the QR decomposition of $X+Y$.
The computational cost of the QR decomposition of an $n\times r$ matrix is $\OO(nr^2)$.

Here, we employ two types of backtracking strategies.
The first one is the standard Armijo linesearch on Riemannian manifolds \cite{absil2008optimization}.
That is, with $s$ as the initial stepsize, the stepsize is determined by the smallest nonnegative integer $m$ such that
\begin{equation}\label{eq:Armijo}
    f(R_{X^{k-1}}(-st^m\grad f(X^{k-1})))-f(X^{k-1}) \le -\sigma st^m\|\grad f(X^{k-1})\|_{X^{k-1}}^2
\end{equation}
holds, where $\sigma,t\in(0,1)$.
Since the standard Armijo backtracking repeatedly computes the retraction, it can become a bottleneck when the retraction is computationally expensive, such as in the case of the QR decomposition.
To address this, \citet{sato2023modified} proposed a method to avoid retraction computations as much as possible during the backtracking.
The reduced Armijo method by \citet{sato2023modified} checks the condition \eqref{eq:Armijo} only when condition
\begin{equation}\label{eq:reduced-Armijo}
    f(X^{k-1}-st^m\grad f(X^{k-1}))-f(X^{k-1}) \le -\sigma st^m\|\grad f(X^{k-1})\|_{X^{k-1}}^2
\end{equation}
is satisfied.
This reduces the number of retraction computations.

For $(n,r)\in\{(25,5),(50,10),(75,15),(100,20)\}$, we conduct comparisons on the problem where $\tilde{A}\in\RR^{n\times n}$ is generated with entries independently following the standard normal distribution, and $A$ is set as $A=\tilde{A}+\tilde{A}^\top$.
The initial point $X^0\in\mathrm{St}(n,r)$ is randomly constructed by \texttt{stiefelfactory} in Manopt \citep{boumal2014manopt}.
To estimate the upper curvature parameter at the initial point, we use a matrix $Y\in\RR^{n\times r}$ whose elements are independently drawn from the standard normal distribution and set
\begin{equation}
    \tilde{L}\coloneqq\frac{2|f(R_{X^0}(Z))-f(X^0)-\tr(\grad f(X^0)^\top Z)|}{\|Z\|_{X^0}^2},
\end{equation}
where $Z$ is the projection of $Y$ onto $T_{X^0}\mathrm{St}(n,r)$.
We use $\sigma=10^{-4},~ t=1/2$, and $s=0.001\tilde{L}$ for the Riemannian gradient methods with Armijo linesearch.
For the AC-PGM, we set $L_0=\theta\tilde{L}$ with $\theta\in\{0.05,0.01,0.005,0.001\}$ and $\alpha=0.6$.
All algorithms are terminated once $\|\grad f(X^k)\|_{X^k}\le10^{-4}$ holds.

\begin{table}[htbp]
    \label{table:vs-Armijo}
    \caption{Computational result of the Riemannian gradient methods. Problem size $(n,r)$, dimension of the manifold (Dim.), computational time (Time), the number of iterations (\#Iter.), the number of retraction computations (\#Retr.) until algorithm termination.}
    \vspace{1em}
    \centering
    \begin{tabular}{cclrrr} \hline
        ($n$,$r$) & Dim. & Algorithm & Time (s) & \#Iter. & \#Retr. \\ \hline\hline
        \multirow{6}{*}{$(25,5)$} & \multirow{6}{*}{110} & Armijo &0.1513&536&8894 \\ \cline{3-6}
        && Reduced Armijo &0.1414&548&4349 \\ \cline{3-6}
        && AC-RGM ($L_0=0.05\hat{L}$) &0.1202&1183&1183 \\ \cline{3-6}
        && AC-RGM ($L_0=0.01\hat{L}$) &0.1150&1183&1183 \\ \cline{3-6}
        && AC-RGM ($L_0=0.005\hat{L}$) &0.1078&1085&1085 \\ \cline{3-6}
        && AC-RGM ($L_0=0.001\hat{L}$) &0.1197&1240&1240 \\ \hline
        \multirow{6}{*}{$(50,10)$} & \multirow{6}{*}{445} & Armijo &2.6362&4877&96497 \\ \cline{3-6}
        && Reduced Armijo &1.3852&6334&21310 \\ \cline{3-6}
        && AC-RGM ($L_0=0.05\hat{L}$) &0.7048&6060&6060 \\ \cline{3-6}
        && AC-RGM ($L_0=0.01\hat{L}$) &0.5673&5122&5122 \\ \cline{3-6}
        && AC-RGM ($L_0=0.005\hat{L}$) &1.0156&8553&8553 \\ \cline{3-6}
        && AC-RGM ($L_0=0.001\hat{L}$) &0.7770&6820&6820 \\ \hline
        \multirow{6}{*}{$(75,15)$} & \multirow{6}{*}{1005} & Armijo &7.9542&6567&131010 \\ \cline{3-6}
        && Reduced Armijo &2.2259&6567&13527 \\ \cline{3-6}
        && AC-RGM ($L_0=0.05\hat{L}$) &3.8431&23501&23501 \\ \cline{3-6}
        && AC-RGM ($L_0=0.01\hat{L}$) &3.2457&19961&19961 \\ \cline{3-6}
        && AC-RGM ($L_0=0.005\hat{L}$) &3.8343&23376&23376 \\ \cline{3-6}
        && AC-RGM ($L_0=0.001\hat{L}$) &3.1253&19018&19018 \\ \hline
        \multirow{6}{*}{$(100,20)$} & \multirow{6}{*}{1790} & Armijo &70.1871&33682&704083
 \\ \cline{3-6}
        && Reduced Armijo &19.9213&33682&72636 \\ \cline{3-6}
        && AC-RGM ($L_0=0.05\hat{L}$) &11.1566&48957&48957 \\ \cline{3-6}
        && AC-RGM ($L_0=0.01\hat{L}$) &6.6723&30059&30059 \\ \cline{3-6}
        && AC-RGM ($L_0=0.005\hat{L}$) &8.6976&39397&39397 \\ \cline{3-6}
        && AC-RGM ($L_0=0.001\hat{L}$) &8.0406&36467&36467 \\ \hline
    \end{tabular}
\end{table}

\begin{figure}[htbp]
    \centering
    \includegraphics[width=.5\linewidth]{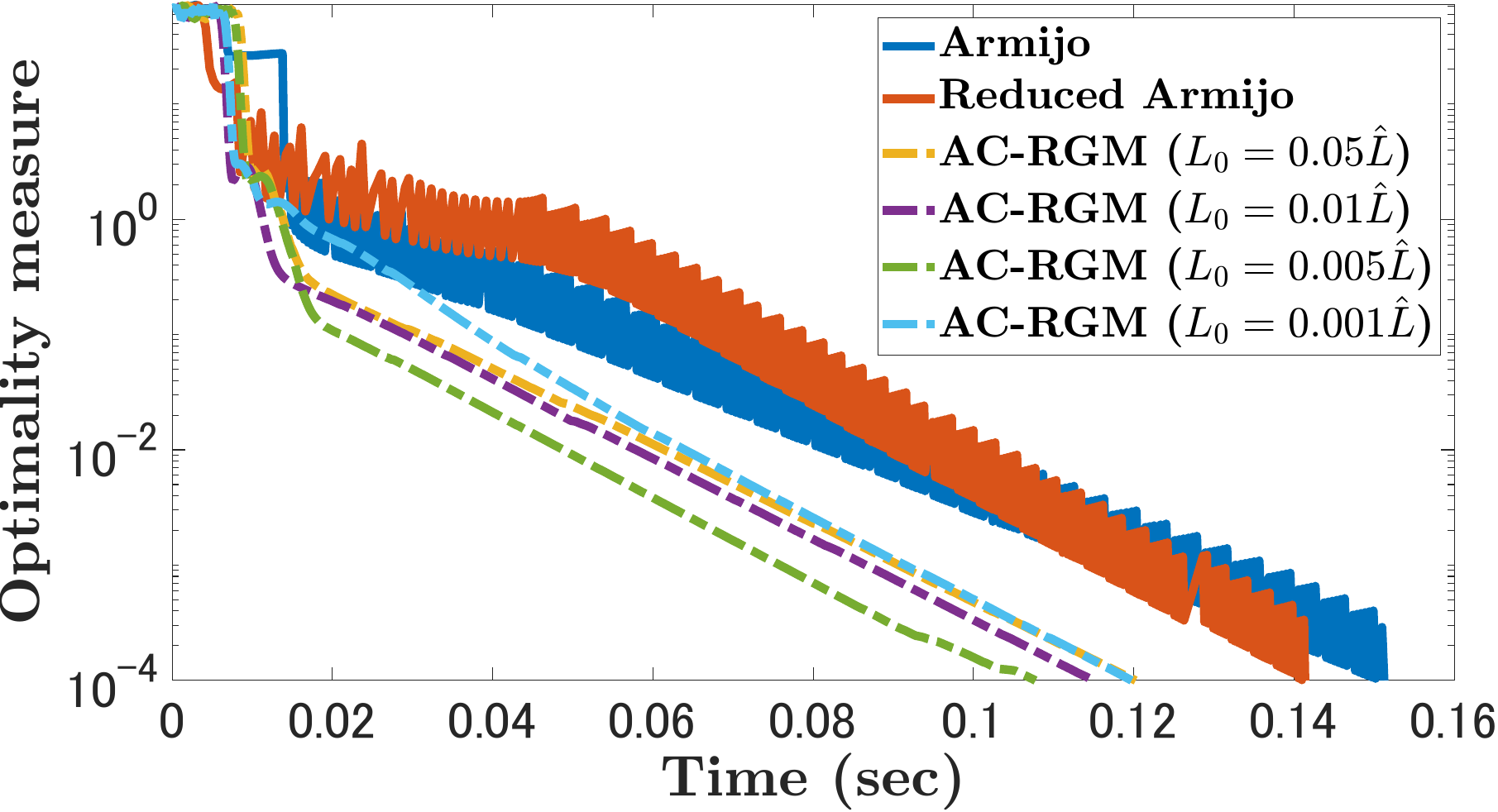} 
    \vspace{1em}
    \includegraphics[width=.5\linewidth]{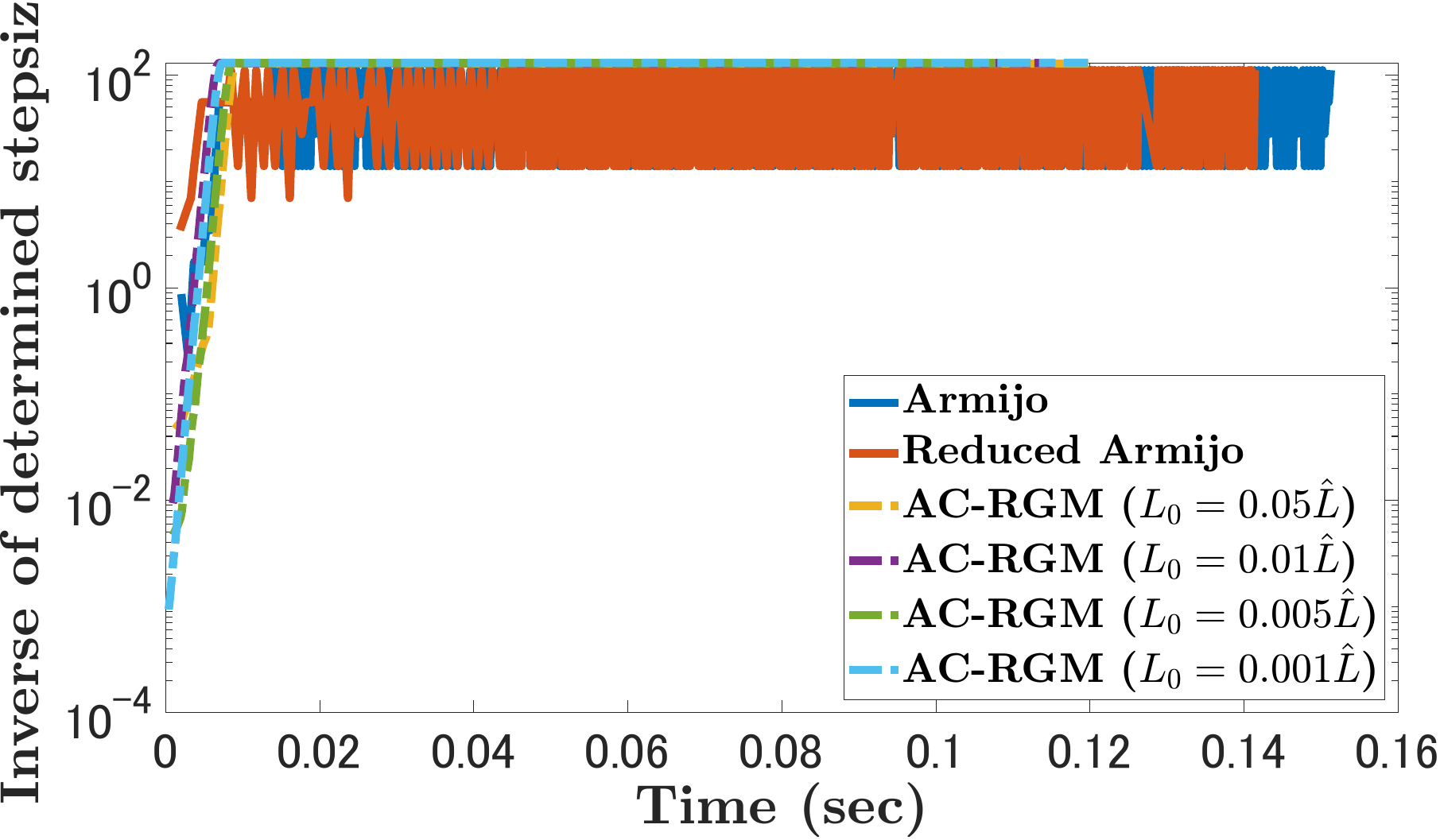} 
    \caption{The convergence behaviors and the inverse of the determined stepsizes of each algorithm for the case $(n,r)=(25,5)$.}
    \label{fig:n25r5}
\end{figure}

\begin{figure}[htbp]
    \centering
    \includegraphics[width=.5\linewidth]{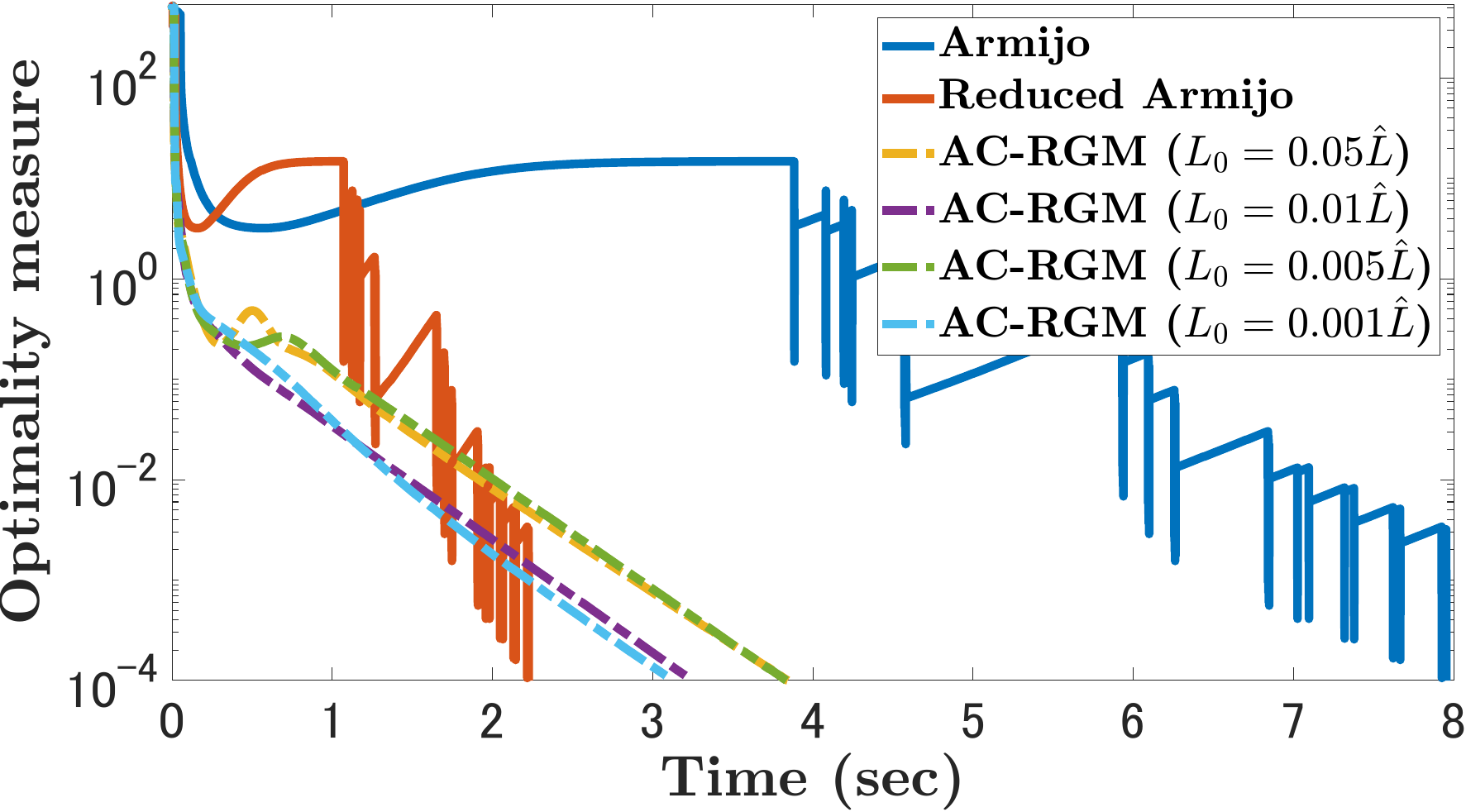} 
    \vspace{1em}
    \includegraphics[width=.5\linewidth]{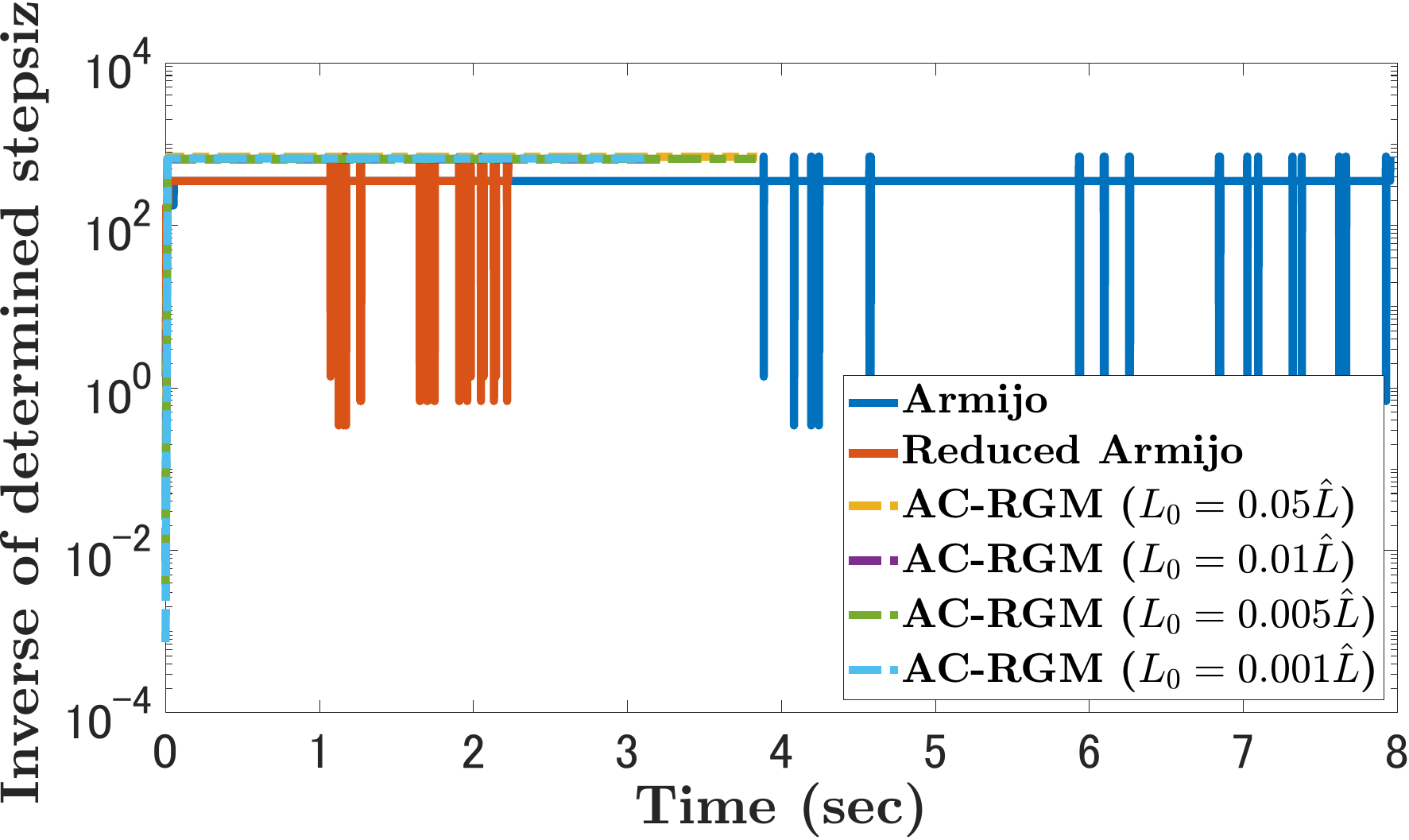} 
    \caption{The convergence behaviors and the inverse of the determined stepsizes of each algorithm for the case $(n,r)=(75,15)$.}
    \label{fig:n75r15}
\end{figure}

Table \ref{table:vs-Armijo} presents the time taken until algorithm termination, the number of iterations, and the number of retraction computations.
It can be observed that the number of retraction evaluations has a significant impact on the computational time.
Except for the case $(n,r)=(75,15)$, the AC-RGM achieves faster convergence than the linesearch-based methods because it does not use linesearch, thereby reducing the number of retraction computations.
We compare the convergence behaviors and the determined stepsizes between the cases where the AC-RGM is faster (Figure \ref{fig:n25r5}) and where it is not (Figure \ref{fig:n75r15}).
From Figures \ref{fig:n25r5} and \ref{fig:n75r15}, it can be seen that the reduced Armijo method outperforms the AC-RGM when it successfully adopts larger stepsizes than the AC-RGM.

\section{Concluding Remarks}\label{sec:conclusion}
In this paper, we first present a proximal gradient method for nonconvex optimization based on the auto-conditioned stepsize strategy proposed by \citet{lan2024projected}.
A simple convergence analysis is conducted.
We also provide a convergence analysis in the presence of the KL property, adaptivity to the weak smoothness, and the extension to the Bregman proximal gradient method.
Furthermore, auto-conditioned conditional gradient and Riemannian gradient methods are also proposed, demonstrating the generality of the auto-conditioned stepsize strategy.

One limitation of this work is that, although our method is parameter-free and linesearch-free, it is not ``adaptive''.
That is, our algorithm is not adaptive to local curvature because it imposes monotonicity on the stepsizes.
As can be seen from Figure \ref{fig:n75r15}, allowing adaptive stepsize selection could further enhance practical performance.
Such a limitation of adaptivity is also shared by existing auto-conditioned methods for nonconvex optimization \cite{lan2024projected,Hoai2024nonconvex,Hoai2025proximal}.
Therefore, developing adaptive linesearch-free methods for nonconvex optimization would be a next challenge.

\section*{Acknowledgments}
Shotaro Yagishita is supported in part by JSPS KAKENHI Grant 25K21158.
Masaru Ito is supported in part by JSPS KAKENHI Grant 25K15010.




\bibliography{reference.bib}
\bibliographystyle{plainnat}

\end{document}